\documentclass[a4paper,12pt,oneside,reqno]{amsart}
\usepackage{amssymb,amsfonts,amsmath,amsthm,amscd}
\usepackage{graphicx, color}
\numberwithin{equation}{section}  

\newcommand{\beq}{\begin{equation}} 
\newcommand{\eeq}{\end{equation}} 
\newcommand{\bea}{\begin{aligned}}
\newcommand{\eea}{\end{aligned}}
\newcommand{\bdm}{\begin{displaymath}}
\newcommand{\edm}{\end{displaymath}}
\newcommand{\barr}{\begin{array}}
\newcommand{\earr}{\end{array}}
\newcommand{\ben}{\begin{enumerate}}
\newcommand{\een}{\end{enumerate}}
\newcommand{\bde}{\begin{description}}
\newcommand{\ede}{\end{description}}

\newtheorem{teor}{Theorem}[section]
\newtheorem{prop}[teor]{Proposition}  
\newtheorem{lem}[teor]{Lemma}  
\newtheorem{cor}[teor]{Corollary}

\newtheorem{rem}[teor]{Remark}
\newcommand{\R}{\mathbb{R}}

\newcommand{\PP}{\mathbb{P}}
\newcommand{\E}{{\mathbb{E}}}

\newcommand{\defi}{\stackrel{\hbox{\tiny def}}{=}}

\newcommand{\be}{\beta}

\newcommand{\de}{\delta}
\newcommand{\dd}{\text{d}}
\newcommand{\ee}{\text{e}}

\newcommand{\w}{\omega}

\newcommand{\s}{\sigma}
\newcommand{\vare}{\varepsilon}

\def\text#1{\;\hbox{\rm #1}\;}
\begin{document}
\title[Genealogy of branching Brownian particles]{The genealogy of extremal particles \\ of Branching Brownian Motion}
\author{L.-P. Arguin, A. Bovier and N. Kistler}
 \address{L.-P. Arguin\\ Courant Institute of Mathematical Sciences \\
New York University \\
251 Mercer St. New York, NY 10012}
\email{arguin@math.nyu.edu}
\address{A. Bovier\\Institut f\"ur Angewandte Mathematik\\Rheinische
   Friedrich-Wilhelms-Uni\-ver\-si\-t\"at Bonn\\Endenicher Allee 60\\ 53115
   Bonn,Germany}
\email{bovier@uni-bonn.de}

\address{N. Kistler\\Institut f\"ur Angewandte Mathematik\\Rheinische
   Friedrich-Wilhelms-Uni\-ver\-si\-t\"at Bonn\\Endenicher Allee 60\\ 53115
   Bonn,
Germany}
\email{nkistler@uni-bonn.de}


\thanks{L.-P. Arguin is supported by the NSF grant DMS-0604869 and partially by the Hausdorff Center for Mathematics, Bonn. 
A. Bovier is partially supported through the German Research Council in the SFB 611. N. Kistler is partially supported by the Deutsche Forschungsgemeinschaft, Contract No. DFG GZ BO 962/5-3. The kind hospitality of Eurandom, Eindhoven, the Hausdorff Center for Mathematics, Bonn, and the Technion, Haifa, where
part of this work has been carried out, are gratefully acknowledged.}

 \date{\today}

\begin{abstract} 
Branching Brownian Motion describes a system of particles which diffuse in space and split into offsprings according to a certain random mechanism. In virtue of the groundbreaking work by M. Bramson on the convergence of solutions of the Fisher-KPP equation to traveling waves \cite{bramson, bramson_monograph}, the law of the rightmost particle in the limit of large times is rather well understood. In this work, we address 
the full statistics of the extremal particles (first-, second-, third- etc. largest). In particular, we prove that in the large $t-$limit, such particles descend with overwhelming probability from ancestors having split either within a distance of order one from time $0$, or within a distance of order one from time $t$. The approach relies on characterizing, up to a certain level of precision, the paths of the extremal particles. As a byproduct, a heuristic picture of Branching Brownian Motion ``at the edge'' emerges, which sheds light on the still unknown limiting extremal process. 
\end{abstract}

\maketitle

\tableofcontents

\section{Introduction} 
Branching Brownian Motion (BBM for short) is a continuous-time Markov branching process which is constructed as follows: a single particle performs a standard Brownian Motion $x$ issued on some probability space $(\Omega, \mathcal F, \PP)$ with $x(0)=0$, which it continues for an exponential holding time $T$ independent of $x$, with $\PP\left[ T > t\right] = \ee^{-t}$. At time $T$, the particle splits independently of $x$ and $T$ into $k$ {\it offsprings} with probability $p_k$, where $\sum_{k=1}^\infty p_k = 1$, $\sum_{k=1}^\infty k p_k = 2$, and $K \defi \sum_{k} k(k-1) p_k < \infty$. These particles continue along independent Brownian paths starting at $x(T)$, and are subject to the same splitting rule, with the effect that the resulting tree $\mathfrak X$ contains, after an elapsed time $t>0$, a random number, $n(t)$, of  particles located at $x_1(t), \dots, x_{n(t)}(t)$. Clearly,  $\E n(t)=e^t$. 
With 
\beq \label{bbm_repr}
u(t, x) \defi \PP\left[ \max_{1\leq k \leq n(t)} x_k(t) \leq x \right],
\eeq
a standard renewal argument, first observed by McKean \cite{mckean}, 
shows that $u(t,x)$ solves
the Kolmogorov-Petrovsky-Piscounov or Fisher (F-KPP) equation, 
\beq \bea \label{kpp_equation}
& u_t = \frac{1}{2} u_{xx} + \sum_{k=1}^\infty p_k u^k -u, \\
& u(0, x)= 
\begin{cases}
1, \; \text{if}\;  x\geq 0,\\
0, \, \text{if}\; x < 0. 
\end{cases}
\eea \eeq 

The F-KPP equation is arguably one of the simplest p.d.e. that admits traveling wave solutions. It is well known that there exists a unique solution satisfying 
\beq \label{traveling_one}
u\big(t, m(t)+ x \big) \to \omega(x), \qquad \text{uniformly in}\;  x,\; \text{as} \; t\to \infty,
\eeq
with the centering term, the {\it front} of the wave, given by
\beq \label{centering_kpp}
m(t) = \sqrt{2} t - \frac{3}{2 \sqrt{2}} \log t, 
\eeq
and $\w(x)$ the unique (up to translation) distribution function which solves the o.d.e. 
\beq \label{wave_pde}
\frac{1}{2} \omega_{xx} + \sqrt{2} \omega_x + \omega^2 - \omega = 0.
\eeq
The leading order of the front has been established through purely analytic means by Kolmogorov, Petrovsky and Piscounov \cite{kpp}. The far more delicate issue of the logarithmic corrections has been settled by Bramson \cite{bramson}, who exploited the probabilistic interpretation of the F-KPP equation in terms of BBM.
Both F-KPP equation and BBM have attracted a great deal of interest ever since: the reader is referred to the very partial list  \cite{aronson_weinberger, aronson_weinberger_two, elworthy_truman_zhao_gaines,elworthy_zhao} for results of analytical flavor, and to \cite{biggins, derrida_spohn, harris, harris_harris_kyprianou, lalley_sellke, neveu} for more probability-oriented work. \\

The large time asymptotic of the maximal displacement of BBM is a paradigm for the behavior of extrema of random fields, which is a classical problem in probability theory. In the case of BBM, the correlations among particles are given in terms of the {\it genealogical distance}: for $i, j\in \Sigma_t \defi \{1, \dots, n(t)\}$ and conditionally upon the branching mechanism, it holds 
\beq
\E\left[ x_i(t) \cdot x_j(t) \right] = Q_{t}(i,j), \label{genealogical}
\eeq
where $Q_t(i,j) \defi \sup\{s\leq t: \; x_i(s)=x_j(s) \}$ is the time to first branching. (In spin glass terminology, $Q_t(i,j)$ is the {\it overlap} between configuration $i$ and $j$). Since $Q_{t}$ can take any value in $[0,t]$, one might expect that the maximal displacement of BBM lies  considerably lower than in the independent, identically distributed (i.i.d.) setting. Somewhat surprisingly, this is not the case: to {\it leading} order, it coincides with that of $\lfloor \ee^t \rfloor$ independent centered Gaussians of variance $t$, in which case the correct centering is well known to be
\beq \label{rem_corr}
r(t) \defi \sqrt{2} t - \frac{1}{2\sqrt{2}} \log t.
\eeq
We will refer henceforth to the Gaussian i.i.d. setting as the {\it Random Energy Model} of Derrida, or REM for short \cite{derrida_rem}.  The law of the maximum of a REM is known to belong to the domain of attraction of the Gumbel distribution,
$G(x) \defi \exp\left(-\ee^{-\sqrt 2 x}\right)$, see e.g. \cite{leadbetter}. Although BBM does not belong to this universality class (it is straightforward to check that $G$ does not solve \eqref{wave_pde}), the distribution of its maximum is still Gumbel-like. Indeed, denoting by
\beq\label{anton.3}
Z(t) \defi \sum_{k=1}^{n(t)} \left( \sqrt{2}t -x_k(t) \right) \exp-\sqrt{2}\left( \sqrt{2}t -x_k(t) \right)
\eeq
the so-called {\it derivative martingale}, Lalley \& Sellke \cite{lalley_sellke} proved that $Z(t)$ converges weakly to a strictly positive 
random variable $Z$, and established the integral representation 
\beq \label{gumbel_like}
\omega(x) = \E\left[ \ee^{- C Z \ee^{-\sqrt{2}x} }\right]
\eeq
(for some $C>0$). This exposes the law of the maximum of BBM as a {\it random shift} of the Gumbel distribution. 

It is also known that the limiting derivative martingale has infinite mean, $\E[Z] = +\infty$. This affects the asymptotics to the right to the extent that 
\beq \label{to_the_right}
1-\omega(x) \sim x \ee^{-\sqrt{2} x}, \quad x\to +\infty, 
\eeq
with $\sim$ meaning that the ratio of the terms converges to a positive constant in the considered limit (see e.g. Bramson \cite{bramson_monograph} and Harris \cite{harris}). Tails of the form \eqref{to_the_right} have recently started to appear in different fields, see e.g. the studies on Spin Glasses with logarithmic correlated potentials \cite{bouchaud_fyodorov, carpentier_ledoussal}; there is thus strong evidence for the existence of a {\it new} universality class different from the Gumbel (which has tail $1- G(x) \sim \ee^{-\sqrt{2}x}$ for $x\to +\infty$).

Contrary to the statistics of the maximal displacement, nothing is known on a rigorous level about the
{\it full} statistics of the extremal configurations (first-, second-, third-, etc. largest) in BBM. 
Such a statistics is completely encoded in the {\it extremal process}, which is the point process associated to the collection
of points shifted by the expectation of their maximum (lower order included), namely 
\begin{equation} \label{extremal process}
\mathcal N_t \defi \sum_{i=1}^{n(t)} \de_{\overline{x_i(t)}}\;, \qquad \quad \overline{x_i(t)} \defi x_i(t)-m(t).
\end{equation}
In fact, it is not even known  whether $\mathcal N_t$ converges to a well defined limit at all, although one can easily see that the  
collection of laws is tight, see Proposition \ref{loc_finite} below. On a non-rigorous level, the situation is only slightly better, 
see Section \ref{towards_extremal_process}  below for a discussion of some recent work by Brunet \& Derrida
\cite{derrida_brunet}. 

The extremal process of the REM is well known to be given in the limit of large times by a Poisson point process
with exponential density $\ee^{-\sqrt{2} x} \dd x$ on $\R$. Given the Gumbel-like behavior \eqref{gumbel_like},
one may (perhaps) be tempted to conjecture that the limiting extremal process of BBM is a {\it randomly shifted} Poisson point process, but 
the work by Brunet \& Derrida mentioned above provides strong evidence against this: BBM seems to belong, as far as statistics 
of extremal particles are concerned, to a new universality class, which is expected to describe the extrema of models ``at criticality'', 
such as the $2-$dim Gaussian free field \cite{bolthausen_deuschel_giacomin}, directed polymers on Cayley trees \cite{derrida_spohn}, or spin glasses 
with logarithmic potentials \cite{bouchaud_fyodorov, carpentier_ledoussal}. \\

In this work we obtain some first rigorous results on the statistics of the extremal particles of BBM. Although we cannot yet characterize the limiting extremal process, a clear picture of BBM at the edge emerges from our analysis, which we believe will prove useful for further studies. 

\section{Main results} \label{main_results}
\subsection{The genealogy of extremal particles}
The major difficulty in the analysis of the BBM stems from the delicate dependencies among particles, which are
due to the continuous branching. A first, natural step towards the extremal process is to study it at the level of the Gibbs measure, which is less sensitive to correlations. The Gibbs measure is the random probability measure on the configuration space $\Sigma_t$ attaching to the particle
$k\in \Sigma_t$ the weight 
\beq\label{anton.4}
\mathcal G_{\be,t}(k) \defi \frac{\exp \be x_k(t)}{\mathcal Z_t(\be)} , \;\text{where} \mathcal Z_t(\be) \defi  \sum_{j=1}^{n(t)}  \exp \beta x_j(t),
\eeq
where $\be>0$ is the inverse temperature. A first study of the Gibbs measure of BBM was carried out by Derrida \& Spohn \cite{derrida_spohn}. Through comparisons with Derrida's GREM, and exploiting the Ghirlanda-Guerra identities  introduced in the context of the Sherrington-Kirkpatrick model \cite{ghirlanda_guerra}, Bovier \& Kurkova \cite{BovierKurkova_II} put on rigorous ground the findings by Derrida \& Spohn, thereby proving in particular that, for $\be> \sqrt{2}$, the law  of the {\it normalized} time to first branching under the product Gibbs measure over the replicated space $\Sigma_t \times \Sigma_t$ converges in distribution to the superposition of two delta functions, 
\beq \label{convergence_gibbs}
\lim_{t\to \infty} \mathcal G_{\be,t} \otimes \mathcal G_{\be, t}\left( \frac{Q_{i,j}(t)}{t} \in \dd x \right)  = c_o \de_0(\dd x) + (1-c_o) \de_1(\dd x).
\eeq
for some $\beta$-dependent $0< c_o < 1$. Hence, the support of the Gibbs measure is restricted to ``almost  uncorrelated'' particles. Since the Gibbs measure favors extremal configurations, one may wonder whether a similar result holds true also at the level
of the extremal process. In this paper, we answer this question in the affirmative. In fact, we prove a {\it stronger} result which concerns the
{\it unnormalized} time to first branching of extremal particles: denoting by
$\Sigma_t(D) \defi \left\{ i\in \Sigma_t: \overline{x_i(t)} \in D \right\}$
the set of particles falling into the subset $m(t)+D$, we have: 

\begin{teor}[The genealogy of extremal particles.] \label{teor_overlaps_extremality}
For any compact $D\subset R$, 
\beq\label{anton.400}
 \lim_{r\to \infty} \sup_{t > 3r} \PP\Big[ \exists i,j \in  \Sigma_t(D): \; Q_t(i, j) \in (r, t-r) \Big] = 0.
\eeq
\end{teor}
Extremal particles thus descend from common ancestors which either branch off "very early", i.e. in the interval $(0,r)$, or ``very late'', i.e. in the interval $(t-r,t)$, in the course of time.
The proof of Theorem \ref{teor_overlaps_extremality} is given in Section \ref{sec_extremality_overlaps} and relies on results about the localization of the paths of the extremal particles which is of independent interest. Such results on the localization of the paths, and the heuristics behind Theorem
\ref{teor_overlaps_extremality} are presented in Section \ref{loc_paths_sec}.  \\

Let us remark in passing that a complete description of the genealogy for a related model, branching Brownian motion with absorption,
has been recently obtained by Berestycki, Berestycki, and Schweinsberg \cite{b2s}. 
In this model, the Brownian particles possess a negative drift and are killed upon reaching zero. 
They show that, for a particular choice of the drift (for which the population is roughly of order $N$)
and an appropriate time scale (of order $(\log N)^3$),
the genealogy of uniformly sampled particles converges as $N\to\infty$ to the Bolthausen-Sznitman coalescent.
It is unlikely that the same limit holds for the genealogy of extremal particles of BBM.
In fact, this is suggested by Theorem \ref{teor_overlaps_extremality},
since it shows that the branching times are strongly concentrated on early and late times
as opposed to the Bolthausen-Sznitman coalescent where all branching times can occur with positive density.

\subsection{Localization of the paths of extremal particles} \label{loc_paths_sec}
Our approach towards the genealogy of particles at the edge of BBM is based on 
characterizing, up to a certain level of precision, the {\it paths} of extremal particles. 
As a first step towards a characterization, we will prove that such paths  
cannot fluctuate too wildly in the {\it upward} direction. 
In order to formulate this precisely, we introduce some notation. 
For $\gamma>0$, we set
\beq\label{anton.1}
f_{t, \gamma}(s) \defi \begin{cases}
                     
		     s^\gamma & 0\leq s \leq \frac{t}{2} \\
		     (t-s)^{\gamma}& \frac{t}{2} \leq s \leq t \ .
                    \end{cases}
\eeq
The {\it upper envelope at time $t$}, denoted $U_{t, \gamma}$, is defined as 
\beq \label{def_upper_envelope}
U_{t, \gamma}(s) \defi \frac{s}{t}m(t) + f_{t, \gamma}(s). 
\eeq
Notice that $U_{t, \gamma}(t) = m(t)$.
\begin{teor}[Upper Envelope] 
\label{teor_upper_envelope} 
Let $0<\gamma<1/2$. Let also $y\in\R$, $\epsilon >0$ be given.
There exists $r_u=r_u(\gamma,y,\epsilon)$ such that for $r\geq r_u$ and for any $t>3r$,
\beq\label{anton.2} \PP\Big[\exists k\leq n(t): x_k(s) > y +U_{t, \gamma}(s),\;\text{for some}\; s \in [r, t-r] \Big]< \epsilon \ .
\eeq
\end{teor}
The picture emerging from Theorem \ref{teor_upper_envelope} is depicted in Figure 1. 

\begin{figure}[here] \label{fig:figure.1}
  \input{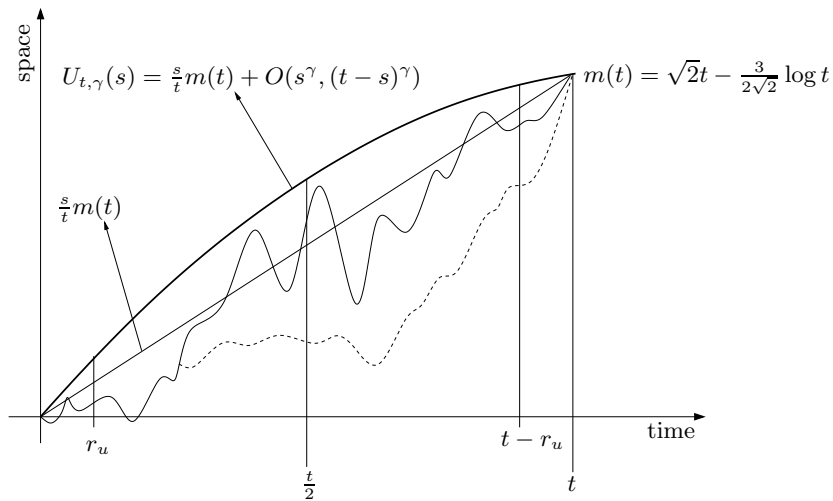}
\caption{The upper envelope}
\end{figure}
Our choice of the upper envelope is not optimal. 
As was pointed out by an anonymous referee, Lalley \& Sellke \cite{lalley_sellke} have shown that $\sqrt{2} t - \max_{k\leq n(t)} x_k(t) \to +\infty$ almost surely.
(This is essentially a consequence of the convergence of the martingale $Z(t)$.)
In particular, this implies that to given $\epsilon >0$, and $r$ large enough,
$$
\PP\Big[\exists k\leq n(t): x_k(s) > y +m(s) +\frac{3}{2\sqrt{2}}\log s,\;\text{for some}\; s >r \Big]< \epsilon \ .
$$
Since $m(s)\leq \frac{s}{t}m(t)$ (for $s>e$), this is an improvement over \eqref{anton.2}, but only for $s$ not too close to $t$.
(For example, it holds for $s\leq t-\delta(t)$ for $\delta\sim (\log t)^{1/\gamma}$.)
For the later times, of order one away from $t$, a finer argument is needed.
A slight refinement of the approach of Bramson \cite{bramson}  can be used to obtain an envelope where $f$ is of  logarithmic order on the entire interval $[r,t-r]$.
It is based on the simple observation that particles that touch  the upper envelope during the interval $[r,t-r]$ would reach at that time values that are so large that  their {\it offsprings} at time 
$t$ could easily jump  to heights  well above the established value of the maximal displacement. 
We choose the upper envelope $U_{t,\gamma}$ since its form is simple throughout the interval and since it requires relatively little work.
It steadily follows from the estimate of the right tail of the maximal displacement, as obtained by Bramson in  \cite[Prop. 3]{bramson}:
\beq \label{bramson_upper_bound_one}
\PP\left[ \max_{k\leq n(t)} x_k(t) \geq m(t)+Y\right] \leq \kappa (1+Y)^2 \exp\left[-\sqrt{2} Y\right],
\eeq
which is valid for $0<Y<\sqrt{t}$ and where $\kappa>0$ is a numerical constant. (Here and henceforth, we 
denote by $\kappa$ a positive numerical constant, not necessarily the same at different occurrences).

%
%

The construction of an upper envelope is very useful for the results on the genealogies, as we shall briefly elucidate.
Remark first that correlations among particles force the front of BBM to lie lower (by a logarithmic factor) than the one of the REM. This has considerable impact
on the finer properties of BBM. A simple calculation reveals already that something unusual is going on: to leading order in $t$, the mean number of exceedances of a level $x$ is given by
\beq \label{rem_order_one}
\E[ \#\{i\leq n(t): x_i(t)>x\}]\sim\frac{\ee^{t}}{\sqrt{2\pi t}} \exp\left( - \frac{x^2}{2t} \right)
\eeq
Note that this quantity is not sensitive to correlations. For the level of the maximum in the REM, $x= r(t)$, this quantity is of order one, as $t\uparrow \infty$, while at the level of the maximum of BBM, 
$x = m(t)$, it is of order $t$! As a consequence, Theorem \ref{teor_overlaps_extremality} does not follow by a straigthforward application of the Markov inequality and classical Gaussian estimates. Theorem \ref{teor_upper_envelope} is a first, fundamental step to overcome this difficulty. In fact, this theorem tells us that 
one  can additionally require that paths of extremal particles never cross the upper envelope (up to an error which can be made as small as wished). 
Precisely, extremal particles perform Brownian motion starting in zero conditioned to  reach certain values at given times.
This can be reformulated in terms of a Brownian bridge of length $t$ starting at 0 and ending at $m(t)$ (omitting lower orders) 
that is not allowed to cross the upper envelope.
This idea is omnipresent in Bramson's  paper \cite{bramson},
 and is used extensively in the present work.
By the very definition of the upper envelope, the situation is equivalent to a
Brownian bridge (starting and ending at time $t$ at zero) which is not allowed to cross the curve $f_{t,\gamma}$ for most of its lifespan. 
The probability of such an event is inversely proportional to the length of the bridge, and will hence
compensate the extra $t$ factor observed above. (In fact, the probability that the Brownian bridge remains below any concave curve such $f_{t, \gamma}$ is 
of the same order of magnitude as if it is required to stay below a straight positive line, as long as $\gamma < 1/2$.)

A second consequence of Theorem \ref{teor_upper_envelope} which plays a crucial role in the proof of Theorem \ref{teor_overlaps_extremality} is a phenomenon we will refer to, by a slight abuse of 
terminology, as {\it entropic repulsion}.
Due to the strong fluctuations of the unconstrained paths, particles which at some point are close to the line $s\mapsto \frac{s}{t}m(t)$ 
have plenty of chances to hit the upper envelope in the remaining time.
One expects that a natural way to avoid this is for the paths to lie {\it well below} the interpolating line for most of the time. 
(In other words, a typical Brownian bridge that is conditioned to lie below the curve $f_{t,\gamma}$ for most of the interval of time
must lie well below $0$; this is not surprising in view of the fact that the conditioned Brownian bridge 
resembles a Bessel bridge \cite{revuz_yor}).
This turns out to be the case, the upshot being that the upper envelope identified in Theorem \ref{teor_upper_envelope} 
can be replaced by a lower ``entropic envelope'', $E$, under which paths of extremal particles lie with overwhelming probability. 
Such a phenomenon is strongly reminiscent of the entropic repulsion encountered in the statistical mechanics of membrane models, see e.g. Velenik's survey \cite{velenik}. 

To formulate this precisely, we need some notation. With $f$ as in \eqref{anton.1} and $\alpha >0$ we denote by {\it entropic envelope} the curve 
\beq \label{def_entropic_envelope}
E_{t, \alpha}(s) \defi \frac{s}{t}m(t) - f_{t,\alpha}(s).
\eeq
Notice  that $E_{t,\alpha} \ll U_{t,\alpha}$.
\begin{teor}[Entropic Repulsion] \label{teor_entropic_envelope}  
Let $D\subset \R$ be a compact set, and $0< \alpha < 1/2$. Set $\overline{D} \defi \sup\{x\in D\}.$ 
For any $\vare>0$ there exists $r_e=r_e(\alpha,D,\vare)$ such that for $r\geq r_e$ and $t>3r$,
\beq\label{anton.6} \bea
& \PP\Big[\exists k\leq n(t): x_k(t) \in m(t)+D\; \text{but}\; \exists_{ s \in [r, t-r]}: x_k(s) \geq \overline{D}+E_{t, \alpha}(s)\;  \Big] < \vare \ .
\eea\eeq
\end{teor} 
The mechanism of entropic repulsion described in Theorem \ref{teor_entropic_envelope} is depicted in Figure 2 below. 
Remark that such path-localizations evidently cannot hold true for times which are close to $0$ or $t$, and this is the reason why ``very old'' resp. ``very recent'' ancestries are indeed not only possible, but, as we believe (see Section \ref{towards_extremal_process} below) also crucial for the peculiar properties of the extremal process of BBM. 

\begin{figure}[here] \label{fig:figure.2}
 \input{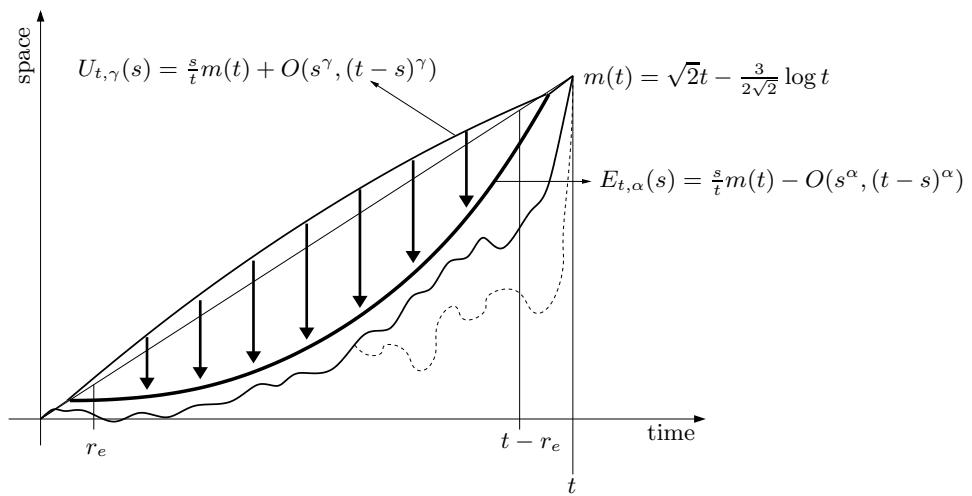}
\caption{Entropic repulsion}
\end{figure}

\begin{rem} 
Energy/entropy considerations provide a straightforward explanation of the mechanism underlying Theorem \ref{teor_entropic_envelope}: 
at any given time $s \in [r, t-r]$ (for $r$ large enough but finite), there are simply not enough particles at heights
$\geq \frac{s}{t} m(t) - f_{t, \alpha}(s)$ for their offsprings to be able to make large jumps allowing them to reach at time $t$ the edge.
\end{rem}

The entropic repulsion is a crucial ingredient in the proof of Theorem \ref{teor_overlaps_extremality}. In fact, consider two extremal particles, say $i$ and $j$, that reach, at time $t$, heights of about $m(t)$, and assume that the common ancestor of these particles branched at times well inside the interval $[0,t]$: for concreteness, assume that $Q_{t}(i,j) = t/3$. By Theorem \ref{teor_entropic_envelope}, the common ancestor of the particles at time $t/3$ lies at heights of order (at most) $\sqrt{2}(t/3) - (t/3)^\alpha$ for some $0<\alpha<1/2$, omitting logarithmic corrections. In order for a descendant, say particle $i$, to be on the edge at time $t$, the ancestor must thus produce a random tree of length $(2/3)t$ where at least one particle makes the unusually high jump $\sqrt{2}(2/3)t+ (t/3)^\alpha$. One can easily check that this is indeed possible: there are to leading order $\exp\left(+ \sqrt{2} (t/3)^\alpha \right)$ particles at levels $\sqrt{2}(t/3) - (t/3)^\alpha$, and the probability of such a big jump in the remaining time interval is of order $\exp\left(-\sqrt{2} (t/3)^\alpha \right)$, the product being thus of order one.  But to have the particle $j$ reach the same levels and overlapping for $t/3$ of its lifetime with $i$ amounts to finding within the same tree of length $(2/3)t$ yet a second particle which makes the unusually high jump. The probability of finding such two particles is to leading order at most $\exp\left(+ \sqrt{2} (t/3)^\alpha \right)  \exp\left(- 2 \sqrt{2} (t/3)^\alpha \right)$, which is vanishing in the limit of large times. Of course, this is valid only to leading order and for a fixed value of the overlap, the nature of the continuous branching compounding the difficulties, but the reasoning is in its essence correct. In fact, for  technical reasons, we need yet another piece of information about
paths of extremal particles: namely that they cannot lie too low, which is again to be expected from an energy/entropy perspective.

\begin{teor}[Lower Envelope] \label{teor_lower_envelope}
Let $D\subset \R$ be a compact set, and $1/2< \beta < 1$. Set $\overline{D} \defi \sup\{x\in D\}$. 
For any $\vare>0$ there exists $r_l= r_l(\beta,D,\vare)$ such that for $r\geq r_l$ and $t>3r$, 
\beq\label{anton.401} \bea
& \PP\Big[\exists k\leq n(t):\, x_k(t)\in m(t)+D, \; \text{but} \exists_{s \in [r,t-r]}:  x_k(s) \leq 
\overline{D}+ E_{t, \beta}(s) \Big] <\vare . 
\eea
\eeq
\end{teor}

Theorems \ref{teor_entropic_envelope} and \ref{teor_lower_envelope} are proven
 in Section \ref{sec_localization_paths}. The two theorems provide 
 an explicitly characterized tube, the space-time region between lower and entropic envelopes, where paths of extremal particles 
spend most of their time with overwhelming probability. 
\begin{cor}[The "tube"]
\label{cor: tube}
Let $D\subset \R$ be a compact set. 
Let $0<\alpha<1/2<\beta<1$.
For any $\vare >0$
there exists $r_1=r_1(\alpha,\beta,D,\vare)$ such that for $r\geq r_1$ and $t>3r$,
\beq \label{eqn: claim}
\bea
&\PP\Big[\forall k\leq n(t)\; \text{such that} \; x_k(t)\in m(t)+D, \; \\
& \hspace{2.5cm} \overline{D}+ E_{t, \beta}(s) \leq x_k(s) \leq  \overline{D} +  E_{t, \alpha}(s) \; \forall_{s \in [r,t-r]}
\Big] \geq 1-\vare \ .
\eea
\eeq
\end{cor}
The proof of the assertion is straightforward from Theorems \ref{teor_entropic_envelope} and \ref{teor_lower_envelope}
taking $$r_1(\alpha,\beta,D,2\vare)=\max\{ r_e(\alpha,D,\vare); r_l(\beta,D,\vare)\}.$$
The image which emerges from the Corollary is depicted in Figure 3 below. 

\begin{figure}[ht] \label{fig:figure.3}
 \input{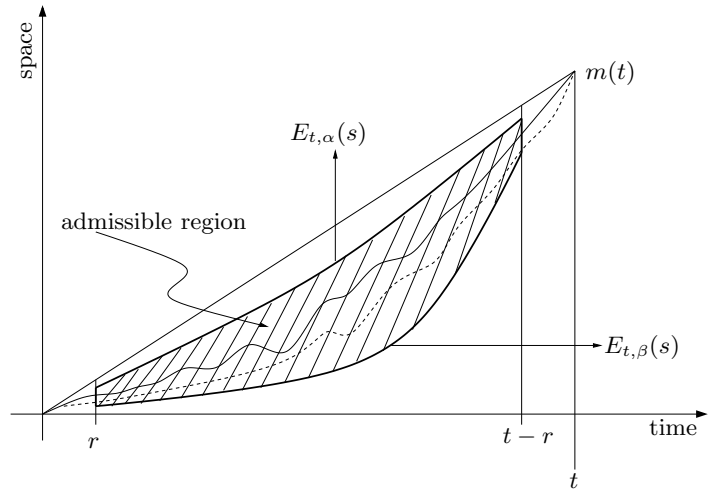}
\caption{The "tube"}
\end{figure}

\subsection{Towards the extremal process of Branching Brownian Motion} \label{towards_extremal_process}

As mentioned in the introduction, it is not known at the time of the writing whether the extremal process 
$\mathcal{N}_t$ of BBM converges as $t\to\infty$ to a well defined point process. 
On the other hand, the tightness of the family $(\mathcal{N}_t)_{t\geq 0}$ is not hard to establish.
\begin{prop}[Local Finiteness]
\label{loc_finite}
For any $y\in\R$ and $\vare>0$, there exists $N=N(\vare, y)$ and $t_0=t_0(\vare,y)$ 
such that for $t\geq t_0$,
$$
\PP[\mathcal{N}_t[y,\infty)\geq N]<\vare\ .
$$
\end{prop}
It is possible to prove the proposition using the localization of the paths described in the above section.
We present here a simpler and more robust proof which has been suggested by an anonymous referee.

\begin{proof}
Suppose the assertion does not hold. Then it is possible to find $y\in \R$, $\epsilon>0$ and 
a sequence of times $(t_N)$ such that
\begin{equation}
\label{eqn: finite1}
\PP[\mathcal{N}_{t_N}[y,\infty)\geq N]\geq \vare ~ \text{ , uniformly in $N$.}
\end{equation}

On the other hand, by the convergence of the law of the maximum \eqref{traveling_one}, for any $\delta>0$, we can find $a_\delta\in \R$
independently of $t$ such that
$$
\lim_{t\to\infty} \PP[\max_k x_k(t)\leq m(t) +a_\delta]\geq 1-\delta\ .
$$
(For example, take $a_\delta$ such that $\omega(a_\delta)=1-\delta$.)

Now pick $\delta=\vare/2$, where $\vare$ is as in \eqref{eqn: finite1}.
Then it must be that, for $N$ large enough, the events
$$
\{\mathcal{N}_{t_N}[y,\infty)\geq N\} \text{ and } \{\max_k x_k(t_N+1)\leq m(t_N+1) +a_{\vare/2}\}
$$
have a non-trivial intersection whose probability is bounded below by $\vare/2$ uniformly in $N$.
We show that this is impossible.
The intersection is included in the event that $N$ particles above $y$ at time $t_N$ produces
an offspring at time $t_N+1$ whose maximum is smaller or equal to $m(t_N+1) +a_{\vare/2}$.
In particular, using the Markov property at time $t_N$,
$$
\bea
&\PP\big[\{\mathcal{N}_{t_N}[y,\infty)\geq N\} \cap \{\max_k x_k(t_N+1)\leq m(t_N+1) +a_{\vare/2}\}\big]  \\
&\leq \prod_{j=1}^N \PP\Big[x_j(t_N)+ \max_{k=1,...,n^j(1)} x_k^{(j)}(1)\leq m(t_N+1) +a_{\vare/2}~\Big|~ x_j(t_N)-m(t_N)\geq y \Big]  \\
&\leq  \Big\{\PP\Big[ \max_{k=1,...,n(1)} x_k(1)\leq m(t_N+1)-m(t_N) -y +a_{\vare/2}\Big]\Big\}^N\ ,
\eea
$$
where $x^{(j)}(1)$, $j=1,...,N$, are iid BBM's at time $1$ with offspring of size $n^j(1)$.
Now since 
$
m(t_N+1)-m(t_N)\leq \sqrt{2}\ ,
$
one has
$$
\bea
\PP\big[\{\mathcal{N}_{t_N}[y,\infty)\geq N\} \cap\max_k x_k(t_N+1)\leq m(t_N+1) +a_{\vare/2}\}\big]\\
\leq  \Big(\PP\Big[ \max_{k=1,...,n(1)} x_k(1)\leq \sqrt{2} - y + a_{\vare/2} \Big]\Big)^N\ .
\eea
$$
The left-hand side goes to zero as $N$ tends to infinity, thereby deriving the contradiction.
\end{proof}

Recent findings by Brunet \& Derrida \cite{derrida_brunet}, largely based on numerical studies, suggest that the limiting extremal process of BBM is an object with striking properties. First, they show evidence that the statistics of the gaps between the leading particles differ from those of the Derrida-Ruelle cascades  \cite{ruelle}, that are known to be the limiting extremal process of the GREM processes, as shown by Bovier \& Kurkova \cite{BovierKurkova}. 
The statistics of the gaps of the cascades are the same as those of Poisson point processes with exponential density.
For instance, the expected distance $D_{t}(n,n+1)$ between the $n^{th}$ and $(n+1)^{th}$ particle is proportional to $1/n$.
For the extremal process of BBM, Brunet \& Derrida argue that  
\beq \label{derrida_brunet_result}
\lim_{t\to \infty}  D_t(n, n+1) \simeq \frac{1}{n} - \frac{1}{n \log n} + \dots,
\eeq
Particles at the edge of BBM are thus more densely packed than in Poisson processes of exponential density; this in particular entails
that the statistics of the particles at the edge of BBM cannot be
given by {\it any} of the cascades. 

Brunet \& Derrida  also  conjecture that the extremal process of BBM retains properties of Poisson point processes with exponential density, namely  the {\it invariance under superposition}: the collection of a finite number of i.i.d. copies of the process, with possibly relative shifts, has the same law for the gaps as the process itself. 

Although we cannot prove {\it any} of the Brunet-Derrida conjectures, our results let them appear rather natural. In fact, 
Theorem \ref{teor_overlaps_extremality} suggests the following picture for the extremal process of BBM, which is depicted in Figure 4 below.

First, the result does not rule out ancestry in the interval $[0, r]$ (in the limit of large times and for large enough $r$): this ``free evolution'' seems to 
naturally generate the derivative martingale appearing in the work of Lalley \& Sellke's \cite{lalley_sellke}. 

Second, ancestry over the period $[t-r, t]$ being also allowed, it is obvious that  {\it small grapes} of length at most $r = O(1)$ (for $t\to \infty$), i.e.
clusters of particles with very recent common ancestor,  appear at the end of the time-interval. This suggests that  particles at the edge of BBM should be
more densely packed than in the  REM case, in agreement with \eqref{derrida_brunet_result}. 

Finally, since the ancestors of the extremal particles evolved independently for ``most'' of the time (in the interval $[r,t-r]$), the extremal process must  exhibit a structure similar to Poisson process with exponential density: this makes plausible the invariance under superposition of the law of the gaps proposed by Brunet \& Derrida.

\begin{figure}[ht] \label{fig:figure.4}
 \input{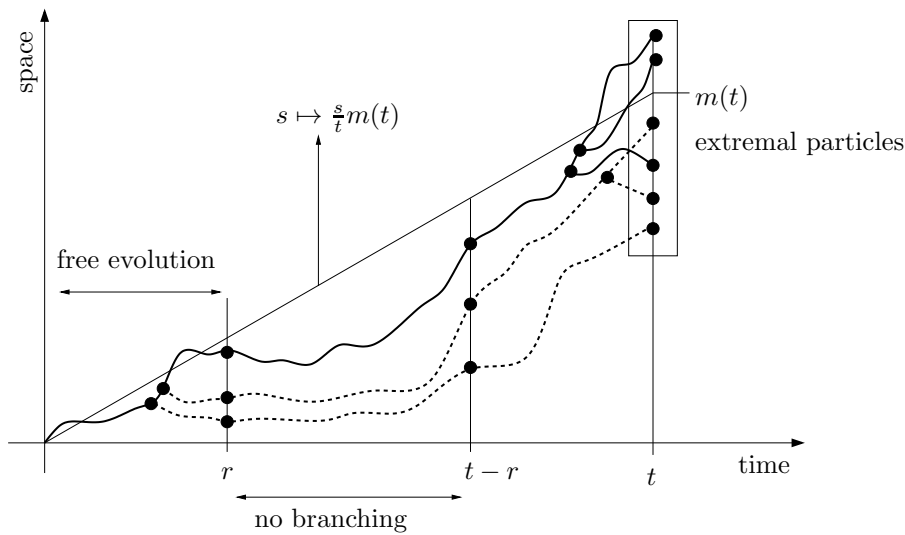}
\caption{Evolution of the system}
\end{figure}

In other words, the limiting extremal process of BBM seems to be given by a certain randomly shifted cluster point process. A rigorous analysis is however technically quite demanding, as it must take into account the self-similarity of BBM which is responsible (in particular) for the onset of the small branches. Unfortunately, this issue is still elusive.

A less ambitious but still interesting goal would be to unravel the poissonian structure which evidently hides behind the limiting extremal process. 
We plan to report on this in future work.

\section{Some properties of Brownian bridge} \label{sec_bb_properties}
We set here some notation and collect useful facts concerning Brownian bridges,
most of which are taken from Bramson \cite[Sect. 2]{bramson_monograph}.

If we denote by $\{x(s), s\geq 0\}$ a standard Brownian motion, it is well known that 
\beq\label{anton.8}
\mathfrak z_t(s) \defi x(s) -\frac{s}{t} x(t), \quad 0\leq s \leq t,
\eeq
defines a new Gaussian process starting and ending at time $t$ in zero, the Brownian bridge.  
We will always assume 
that both $x$ and $\mathfrak z_t$ are defined on $C([0,t], \mathcal B)$, the space of continuous functions 
on $[0,t]$ endowed with its Borel $\s-$algebra. 
We denote by $P^0$ the corresponding measure of $\mathfrak{z_t}$ on $C([0,t], \mathcal B)$.

\begin{lem} \label{prop_bbridge} The Brownian bridge $\mathfrak z_t(s) = x(s)-\frac{s}{t} x(t)$ has the following properties:
\begin{itemize}
\item[1.] $\mathfrak z_t (s)$ is a strong Markov process. 
\item[2.] $\mathfrak z_t(s)$ for $0\leq s \leq t$ is independent of $x(t)$.
\end{itemize}
\end{lem}

If $l$ is a curve $l:[s_1, s_2]\rightarrow \R$, then $B_{l}$ (or $B_l[s_1, s_2]$ in case of ambiguity) will denote the set of paths lying strictly above $l$ on $[s_1, s_2]$. 
Similarly, $B^l$ will denote the set of paths lying strictly below $l$.
The following monotonicity property plays a crucial r\^ole in Section \ref{sec_localization_paths}. 

\begin{lem} \cite[Lemma 2.6]{bramson_monograph} \label{monotonicity} Assume that the curves $l_1, l_2$ and $\Lambda$ satisfy
$l_1(s) \leq l_2(s) \leq \Lambda(s)$ for all, $s\in [0,t]$, and that $P^0\left[B_{l_2}[0,t] \right]> 0$. 
then 
\beq \label{below_lambda}
P^0\left[B^\Lambda \big| B_{l_1} \right] \geq P^0\left[B^\Lambda \big| B_{l_2}\right]
\eeq
and
\beq \label{above_lambda}
P^0\left[B_\Lambda \big| B_{l_1} \right] \leq P^0\left[B_\Lambda \big| B_{l_2}\right].
\eeq

\end{lem}

The following lemma yields uniform bounds on 
conditional probabilities of a Brownian bridge to stay below certain concave curves.  
\begin{lem} \cite[Lemma 2.7]{bramson_monograph} \label{more_than_fluctuations} Set
 \beq\label{anton.9}
                 \Lambda_t(s) = \begin{cases}
                                 C s^\vare & \text{for} \quad 0\leq s \leq t/2, \\
				C(t-s)^\vare & \text{for} \quad t/2 \leq s \leq t, 
                                \end{cases}
                \eeq
where $\vare > 1/2$ and $C>0$. Then, 
\beq \label{more_uniform}
P^0\left[ B^{\Lambda_t}[r, t-r] \Big| B_0[r, t-r] \right] \to 1 \; \text{uniformly in}\; t>3r \; \text{as}\; r\to \infty.
\eeq
More precisely, for a fixed constant, $a>0$, and $\de \defi 2\vare-1>0$, 
\beq \label{more_bound}
P^0\left[ B^{\Lambda_t}[r, t-r] \Big| B_0[r, t-r] \right] \geq 1- 2 a C \sum_{k= r}^\infty k \ee^{-C k^\de}.
\eeq
\end{lem}

Lemma \ref{more_than_fluctuations} allows  to compare probabilities that Brownian bridge  
hits curves which are close to one another, see Section \ref{sec_localization_paths} below. We will also need
to compute probabilities that Brownian bridge stays below linear functions, which is of course straightforward (the only minor issue  
is that such linear functions will possibly interpolate between points which are allowed to grow with time, see Section \ref{sec_extremality_overlaps}).  

\begin{lem}
\label{lem: BB estimate}
Let $Z_1,Z_2\geq 0 $ and $r_1,r_2\geq 0$. 
Then for $t>r_1+r_2$,
\beq\label{anton.10}
\PP\left[\mathfrak z_t(s) \leq (1 - \frac{s}{t} )Z_1 +\frac{s}{t} Z_2, r_1\leq s \leq t-r_2 \right]\leq 
\frac{2}{t-r_1-r_2}\prod_{i=1,2}\left(Z(r_i)+\sqrt{r_i}\right)\ ,
\eeq
where $Z(r_1)\defi (1 - \frac{r_1}{t} )Z_1 +\frac{r_1}{t} Z_2$ and
$Z(r_2)\defi\frac{r_2}{t} Z_1+(1 - \frac{r_2}{t} )Z_2$.
\end{lem}

\begin{proof}
We suppose that $r_1,r_2>0$. The result for $r_1=0$ or $r_2=0$ follows by continuity.
We first condition on the value of $\mathfrak z_t$ at $s=r_1$ and $s=t-r_2$.
The probability
\begin{equation}
\label{eqn: prob estimate 1}
\PP\left[\mathfrak z_t(s) \leq (1 - \frac{s}{t} )Z_1 +\frac{s}{t} Z_2, r_1 \leq s \leq t-r_2 \Big| \mathfrak z_t(r_1)=x_1, \mathfrak z_t(t-r_2)=x_2\right]
\end{equation}
is the same as the probability of the event that a Brownian bridge of length $t-r_1-r_2$ starting at $x_1$ and ending at $x_2$
lies below the straight line joining $Z(r_1)\defi(1 - \frac{r_1}{t} )Z_1 +\frac{r_1}{t} Z_2$ and $Z(r_2)\defi\frac{r_2}{t} Z_1 +\left(1-\frac{r_2}{t}\right) Z_2$, in the notation of the statement.
It is not hard to show (see, e.g., \cite{scheike}) that this probability is exactly
\beq\label{anton.101}
1-\exp\left\{\frac{-2}{t-r_1-r_2}\left(Z(r_1)-x_1\right)\left(Z(r_2)-x_2\right)\right\} \ .
\eeq
In particular, on the event $\left\{\mathfrak z_t(s) \leq (1 - \frac{s}{t} )Z_1 +\frac{s}{t} Z_2\right\}$, 
both $Z(r_1)-x_1$ and $Z(r_2)-x_2$ are non-negative, and (\ref{anton.101}) is smaller than
\begin{equation}
\label{eqn: prob estimate 2}
\frac{2(Z(r_1)-x_1)(Z(r_2)-x_2)}{t-r_1-r_2}=\frac{2}{t-r_1-r_2}\Big(Z(r_1)Z(r_2)-x_1Z(r_2)-x_2Z(r_1) +x_1x_2\Big)\ .
\end{equation}
Therefore the left-hand side of (\ref{anton.10}) is smaller than
\begin{equation}
\label{eqn: prob estimate 3}\bea
\int_{-\infty}^{Z(r_1)}\int_{-\infty}^{Z(r_2)}& \frac{2}{t-r_1-r_2}\Big(Z(r_1)Z(r_2)-x_1Z(r_2)-x_2Z(r_1) +x_1x_2\Big) 
\\&\times \PP\left[\mathfrak z_t(r_1)\in \dd x_1, \mathfrak z_t(t-r_2)\in \dd x_2\right] \ .\eea
\end{equation}
The integral over the first term is smaller than  $\frac{2}{t-r_1-r_2}Z(r_1)Z(r_2)$. 
By the Cauchy-Schwarz inequality, the second term is smaller than 
\beq\label{anton.102}
\frac{2Z(r_2)}{t-r_1-r_2}\left(\int_{-\infty}^{\infty}x_1^2 \ \PP\left[\mathfrak z_t(r_1)\in \dd x_1\right] \right)^{1/2}
=\frac{2Z(r_2)\sqrt{r_1(1-r_1/t)}}{t-r_1-r_2}\leq \frac{2Z(r_2)\sqrt{r_1}}{t-r_1-r_2}\ ,
\eeq
since $\E \mathfrak z_t(r_1)^2= r_1(1-r_1/t)$.
Similarly, the third term of \eqref{eqn: prob estimate 3} is  bounded by
\beq\label{anton.1002}
\frac{2Z(r_1)\sqrt{r_2}}{t-r_1-r_2} \ .
\eeq
Finally, the fourth term is bounded, by the Cauchy-Schwarz inequality, by
\beq\label{anton.103}
\frac{2}{t-r_1-r_2} \left[\int_{-\infty}^\infty x_1^2  \PP\left[\mathfrak z_t(r_1)\in \dd x_1\right]\right]^{1/2} \left[\int_{-\infty}^\infty x_2^2  \PP\left[\mathfrak z_t(t-r_2)\in \dd x_2\right]\right]^{1/2}\leq \frac{2\sqrt{r_1r_2}}{t-r_1-r_2}\ .
\eeq
The claimed bound is obtained by regrouping the four terms.
\end{proof}

\section{The genealogy of extremal particles - proofs} \label{sec_extremality_overlaps}
In this section, we  give a proof of Theorem \ref{teor_overlaps_extremality} based on the localization of the paths.

\begin{proof}[Proof of Theorem \ref{teor_overlaps_extremality}]
Let $D\subset \R$ be a compact set. Then  there exist  $\underline{D}\leq \overline{D}\in \R$, such that   $D\subseteq [\underline{D},\overline{D}]$.
To prove the theorem, we need to find $r_o=r_o(D,\vare)$ and $t_o=t_o(D,\vare)$ such that for $r\geq r_o$ and $t>\max\{t_o,3r\}$,
\beq
\label{eqn: to prove}
\PP\left[ \exists i,j \leq n(t): x_i(t), x_j(t) \in m(t)+D \; \text{and}\; Q_t(i, j) \in [r, t-r] \right] < \vare \ .
\eeq
Corollary \ref{cor: tube} shows the existence of $r_1=r_1(D,\vare)$ such that, with probability at least $1-\vare$,
the extremal particles reaching  $D$ at time $t$ satisfy for $t>3r_1$,
\beq \label{tube}
\overline{D}+\frac{s'}{t} m(t) - f_{t, \beta}(s)   \leq x_i(s') \leq \overline{D}+\frac{s'}{t}m(t)- f_{t, \alpha}(s') 
\quad \forall { r_1 \leq s' \leq t-r_1} \ .
\eeq
Here and henceforth,  $ \Xi_{D, t}$ will denote the set of paths, $x(s')$, which satisfy \eqref{tube}, and for which
 $x(t) \in m(t)+D$. We will denote by $ \Xi_{D, t}^{[s,t-r_1]}$ the greater set of paths where the inequalities
 \eqref{tube} are satisfied for all  $s'\in [s,t-r_1]$ for some $s\geq r_1$, and again $x(t) \in m(t)+D$.

Let $K \defi \sum_{k} p_k k(k-1)$ where $\{p_k\}$ is the offspring distribution. 
We  need a straightforward generalization of Lemma 10 of  \cite{bramson}.
It expresses the expected number of pairs of particles of BBM whose path satisfies some conditions, say $\Xi_{D, t}$,
\begin{equation}
\label{eqn: formula moment Bramson}
 \bea
&\E\left[\#\left\{(i,j): x_i, x_j \in \Xi_{D, t} \  i\neq j \right\}\right]\\
 &\qquad = K \ee^t \int_0^{t} \ee^{t-s} \dd s \int_{-\infty}^\infty \dd \mu_s(y) \PP\left[x\in \Xi_{D, t} \mid x(s) = y \right] 
\PP\left[x\in \Xi_{D, t}^{[s, t-r_1]} \mid x(s) = y \right] ,
\eea
\end{equation}
where $\mu_s$ is the Gaussian measure of variance $s$.
In the case where the event considered includes a condition on $Q_t(i,j)$, 
the integral on the branching time $\dd s$ changes and we have
\begin{equation}
\label{eqn: formula moment}
 \bea
&\E\left[\#\left\{(i,j)\ i\neq j: x_i, x_j \in \Xi_{D,t}, \; Q_t(i,j) \in [r, t-r]\right\}\right]\\
 &\qquad = K \ee^t \int_r^{t-r} \ee^{t-s} \dd s \int_{-\infty}^\infty \dd \mu_s(y) \PP\left[x\in \Xi_{D, t} \mid x(s) = y \right] 
\PP\left[x\in \Xi_{D, t}^{[s, t-r_1]} \mid x(s) = y \right] .
\eea
\end{equation}
We will show the existence of a $r_o=r_o(D,\epsilon)$ and $t_o(D,\epsilon)$ such that, for $r>r_o$ and $t>\max\{t_o,3r\}$, 
the right-hand-side is smaller than $\epsilon$.
This will imply \eqref{eqn: to prove} by Markov's inequality and Corollary \ref{cor: tube} (provided we take $r_o>r_1$).

The idea is to bound the term $\PP\left[x\in \Xi_{D, t}^{[s, t-r_1]} \mid x(s) = y \right]$ uniformly in $y$. 
Since $s'\mapsto \overline{D}+\frac{s'}{t}m(t) -f_{t, \alpha}(s')$ is a convex function that equals $m(t)+\overline{D}$ at time $t$, the event
\beq\label{anton.105}
\left\{x(s')\leq \overline{D}+ \frac{s'}{t}m(t)- f_{t, \alpha}(s'),\,\forall_{ s\leq s' \leq t-r_1}\right\}
\eeq
is contained in  the event where $x(s')$ lies below the straight line joining $\overline{D}+\frac{s}{t}m(t) -f_{t, \alpha}(s)$
at time $s$ to $m(t)+\overline{D}$ at time $t$, i.e.
\beq\label{anton.11}
 \left\{x(s')\leq \frac{(1-\frac{s}{t})m(t)+f_{t, \alpha}(s')}{t-s}\{s'-s\} + \overline{D}+\frac{s}{t}m(t) -f_{t, \alpha}(s), \,\forall_{s\leq s' \leq t-r_1}\right\}. 
\eeq
\begin{figure}[here] \label{fig:figure.5}
\input{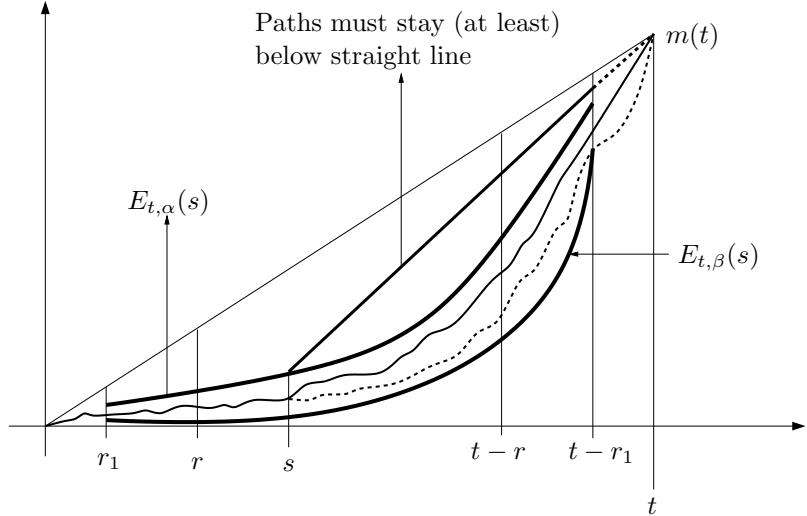}
\caption{Reference for Theorem \ref{teor_overlaps_extremality}}
\end{figure}

\begin{center}

\end{center}

We write $a= \overline{D}+\frac{s}{t}m(t) -f_{t, \alpha}(s) -y$ and  $b=\frac{(1-\frac{s}{t})m(t)+f_{t, \alpha}(s)}{t-s}$ .
Subtracting $x(s)$ and  $\frac{s'-s}{t-s} x(t)$ from $x(s')$, and shifting the time $s'$ by $s$, we get that 
\beq\label{anton.12} \bea
& \PP\left[x\in \Xi_{D, t}^{[s, t-r_1]} \mid x(s) = y \right]\\&\leq
\PP\Big[x(s')-\frac{s'}{t-s} x(t-s) \leq a+bs' -\frac{s'}{t-s} x(t-s) ,  \forall_{ 0\leq s'\leq t-s-r_1}, x(t-s)\in m(t)-y+D  \Big]\ .
\eea \eeq
Since $D\subseteq [\underline{D},\overline{D}]$, this is bounded above by
\beq\label{anton.13} \bea
& \PP\Big[  x(s') -\frac{s'}{t-s} x(t-s)  \leq (1-\frac{s'}{t-s}) Z_1 +\frac{s'}{t-s} Z_2 , \forall_{0\leq s'\leq t-s-r_1} , \ x(t-s)\geq m(t)-y+\underline{D}  \Big],
\eea \eeq
for $Z_1=\overline{D}+\frac{s}{t}m(t) -f_{t, \alpha}(s)-y$ and $Z_2=\overline{D}-\underline{D}$. By the independence property of Brownian bridge, this is simply
\begin{equation}
\label{eqn: estimate prob 1}
\PP\left[\mathfrak z_{t-s}(s') \leq  (1-\frac{s'}{t-s}) Z_1 +\frac{s'}{t-s} Z_2, \, \forall_{0\leq s'\leq t-s-r_1}\right] \PP\left[ x(t-s)\geq m(t)-y+\underline{D}  \right]\ .
\end{equation}
By Lemma \ref{lem: BB estimate}, the Brownian bridge probability is smaller than
\begin{equation}
\label{eqn: estimate prob}
\frac{ 2 \ Z_1}{t-s-r_1}\left(\frac{r_1}{t-s} Z_1 +\left(1-\frac{r_1}{t-s}\right)Z_2 + \sqrt{{r_1}}\right)\ .
\end{equation}
Since $x(s)=y$ lies between the entropic and lower envelope by equation \eqref{tube}, we must have 
$\overline{D}+\frac{s}{t}m(t)-  f_{t, \beta}(s) \leq x(s) \leq \overline{D}+\frac{s}{t}m(t)  -f_{t, \alpha}(s)$, which implies
\beq\label{anton.106}
0\leq Z_1 \leq f_{t, \beta}(s) -f_{t, \alpha}(s) \leq \kappa f_{t, \beta}(s) 
\eeq
for some constant $\kappa>0$ independent of $t$ and $r$. 
(For the rest of the proof, we use $\kappa$ for a generic term whose value might change but that does not depend on $r$ and $t$, and might depend on $D$).

We now precise the choice of $r_o(D,\epsilon)$ to make \eqref{eqn: estimate prob} small. 
Since $r\leq t-s$ in the integral \eqref{eqn: formula moment}, it holds that $\frac{1}{t-s-r_1}\leq \frac{2}{t-s}$ for the choice $r>2r_1$. Hence we require $r_o>2r_1$.
Moreover $\frac{r_1Z_1}{t-s}\leq \frac{r_1\kappa f_{t, \beta}(s)}{t-s}\leq \kappa r_1 (t-s)^{\beta-1}$ by the estimate on $Z_1$. Since $r\leq t-s$, we can pick 
$r_o^{1-\beta}> r_1$ so that $\frac{r_1Z_1}{t-s}<1$. 
We also require that $\sqrt{r_o} \geq \max\{Z_2,1\} $.
With these choices as well as \eqref{anton.106}, \eqref{eqn: estimate prob} can be made smaller than
\beq\label{anton.200}
\frac{\kappa \sqrt{r} Z_1}{t-s}\leq \frac{\kappa \sqrt{r} \ f_{t, \beta}(s) }{t-s}\ .
\eeq

Because $y\leq \frac{s}{t}m(t)-f_{t, \alpha}(s)$, the second term of \eqref{eqn: estimate prob 1} is, by an estimate of the Gaussian density, smaller than
\beq\label{anton.108}
\kappa \frac{e^{-(t-s)}e^{\frac{3}{2}\frac{t-s}{t}\log t} e^{-\sqrt{2} f_{t, \alpha}(s)}}{(t-s)^{1/2}}\ .
\eeq
Putting this together, we obtain the bound uniform in $y$ we were looking for
\begin{equation}
\label{eqn: unif x}
\PP\left[x\in \Xi_{D, t}^{[s, t-r_1]} \mid x(s) = y \right]\leq \kappa \sqrt{r}\ \frac{e^{-(t-s)}e^{\frac{3}{2}\frac{t-s}{t}\log t} f_{t, \beta}(s)   e^{-\sqrt{2} f_{t, \alpha}(s)}}{(t-s)^{3/2}}\ .
\end{equation}
And the right-hand side of \eqref{eqn: formula moment} is bounded above by
\beq
\label{eqn: overlap integral to bound}
 \kappa \ \ee^t  \PP\left[x\in \Xi_{D, t}\right] \ \sqrt{r} \int_r^{t-r} \frac{e^{\frac{3}{2}\frac{t-s}{t}\log t}  f_{t, \beta}(s)  e^{-\sqrt{2} f_{t, \alpha}(s)}}{(t-s)^{3/2}} \dd s\ .
\eeq

The term $\ee^t  \PP\left[x\in \Xi_{D, t}\right]$  is of order $r_1$, uniformly for $t\geq 3r_1$,
which is ensured if $t>3r$.
Indeed, we  have
\beq\label{anton.109}
 \PP\left[x\in \Xi_{D, t}\right]\leq \PP\left[x(s)\leq \frac{s}{t} m(t) + \overline{D}\, \forall_{ r_1 \leq s \leq t-r_1}, x(t)\in m(t) +D \right]\ ,
\eeq
which, using the bounds on $D$, is smaller than
\beq\label{anton.110}
\PP\left[\mathfrak z_t(s) \leq \overline{D},\, \forall_{ r_1 \leq s \leq t-r_1}\right] \PP\left[ x(t)\geq m(t) +\underline{D} \right]\ .
\eeq
By Lemma \ref{lem: BB estimate}, the first term is smaller than $\kappa\frac{r_1}{t-2r_1}$,
and, by the choice of $t$, smaller than $\kappa \frac{r_1}{t}$.
Since $\PP\left[ x(t)\geq m(t) +\underline{D} \right]\leq \kappa  t e^{-t} $,
the claim is proven.

It remains to show that \eqref{eqn: overlap integral to bound} can be made arbitrarily small by taking $r_o$ large.
To do so, we split the domain of integration into the intervals $[r,t/2]$ and $[t/2, t-r]$.
The integral over the first interval is smaller than
\beq\label{anton.111}
\kappa \ r^{3/2}  \int_r^{\infty} s^\beta e^{- (\sqrt{2}) s^\alpha}\dd s\ .
\leq \kappa \kappa' r^{3/2} e^{- r^\alpha},
\eeq
for some $\kappa'$ depending on $\alpha$; this tends to  zero as $r\to\infty$.
On the second interval, we can perform the change of variable $s \to t-s$ to get
\beq\label{anton.14}
\bea
\kappa \  r^{3/2} \int_r^{t/2}  \frac{e^{\frac{3}{2}\frac{s}{t}\log t}s^{\beta}e^{-\sqrt{2} s^\alpha}}{s^{3/2}} \dd s\ .
\eea
\eeq
The integration domain  can again be split into  $[r,t^\delta]$ and $[t^\delta,t/2]$ ($0< \de < 1$). 
It is easy to see that for $t$ large enough the first part is bounded
by 
\beq\label{anton.15}
\kappa \ r^{3/2} \int_r^\infty \frac{s^\beta e^{- \sqrt{2} s^\alpha}}{s^{3/2}}  \dd s \leq \kappa r^{3/2} e^{- r^\alpha}
\eeq
which goes to zero as $r\to\infty$. The second part is smaller than
\beq\label{anton.16}
\kappa \  r^{3/2} t^{3/4}\int_{t^\delta}^{t/2}\frac{s^\beta e^{-   \sqrt{2} s^\alpha}}{s^{3/2}} \dd s \leq \kappa \ r^{3/2} \ t^{7/4} e^{-  t^{\alpha\delta}}\leq
 \kappa \ t^{13/4}e^{- t^{\alpha\delta}},
\eeq
which can be made arbitrarily small by defining $t_o$ large enough.
This concludes the proof of the theorem.
\end{proof}

\section{Localization of paths} \label{sec_localization_paths}
The localization of the paths is based on the determination of an upper envelope, Theorem \ref{teor_upper_envelope},
that the paths of the extremal particles do not reach with overwhelming probability. 
Such an upper envelope allows to 
characterize a tube in which paths of extremal particles spend most of their time. This will be done 
by first identifying in Section \ref{sec_entropic_repulsion} the entropic envelope, hence providing a proof of Theorem \ref{teor_entropic_envelope}. 
We will then identify in Section \ref{sec_lower_envelope} the lower envelope, providing a proof of Theorem \ref{teor_lower_envelope}. 

\subsection{The upper envelope} 
\label{proof_locfin} 
The proof of Theorem \ref{teor_upper_envelope} is an elementary application of the estimate \eqref{bramson_upper_bound_one}.
Recall the definition of $U_{t,\gamma}$ and $f_{t,\gamma}$ given
in Section \ref{loc_paths_sec}. We first prove that the maximum of the process at integer times cannot cross the upper envelope $U_{t,\gamma}$. 
Gaussian estimates extend it to all times.  For convenience, we shall find an upper bound to the probability
\beq\label{anton.120}
\PP\Big[\exists k\leq n(t): x_k(s) > y +m(s)+f_{t, \gamma}(s)\;\text{for some}\; s \in [r, t-r] \Big]\ .
\eeq
The above probability is the probability of Theorem \ref{teor_upper_envelope}
where $U_{t,\gamma}(s)=\frac{s}{t}m(t)+f_{t,\gamma}(s)$ is replaced by $m(s)+f_{t, \gamma}(s)$.
An upper bound for the above probability readily yields an upper bound 
for the probability of the theorem, since
\beq\label{anton.121}
\frac{s}{t}m(t)> m(s)  \ ,
\eeq
for $s>e=2.718...$.
(The above inequality reduces to $
\frac{\log s}{s}> \frac{\log t}{t}$, 
and $\frac{\log x}{x}$ is decreasing for $x>e$).
Throughout this section, we write $\lceil s \rceil$ for the smallest integer greater or equal to $s$. 
\begin{lem}
\label{lem: integer time}
Let $0<\gamma<1/2$.
For any $y\in\R$ and $\epsilon>0$, there exists $r'(\gamma,y,\epsilon)$ such that for $r>r'(\gamma,y,\epsilon)$ and $t>3r$, 
\beq\label{anton.122}
\begin{aligned}
\PP\Big[\exists_{s\in [r,t-r]} : \max_{k\leq n(\lceil s \rceil)}x_k(\lceil s \rceil)\geq y+m(\lceil s \rceil)+f_{t,\gamma}(\lceil s \rceil)\Big]<\epsilon
\end{aligned}
\eeq
\end{lem}
\begin{proof}
The event $\Big\{\exists_{s\in [r,t-r]} : \max_{k\leq n(\lceil s \rceil)}x_k(\lceil s \rceil)\geq y +m(\lceil s \rceil)+f_{t,\gamma}(\lceil s \rceil)\Big\}$ is the union 
over $j=\lceil r \rceil, \lceil r \rceil+1, ... ,\lceil t-r\rceil$ of the events
\beq\label{anton.123}
\Big\{ \max_{k\leq n(j)}x_k(j)\geq y +m(j)+f_{t,\gamma}(j)\Big\}\ .
\eeq
We choose $r$ large enough so that $y+f_{t,\gamma}(r)$ is positive.
Applying \eqref{bramson_upper_bound_one} on this event for $t=j$ and $Y=y+f_{t,\gamma}(j)$, we get using the symmetry of the curve $f_{t,\gamma}$,
\beq\label{anton.201}
\begin{aligned}
&\PP\Big[\exists_{ s\in [r,t-r]} : \max_{k\leq n(\lceil s \rceil)}x_k(\lceil s \rceil)\geq y+m(\lceil s \rceil)+f_{t,\gamma}(\lceil s \rceil)\Big]
\\
&\hspace{4cm}\leq 
2\kappa\sum_{j=\lceil r \rceil}^{\lceil t/2 \rceil+1} (1+(j^\gamma+y))^2\exp(-\sqrt{2}(j^\gamma+y))\ .
\end{aligned}
\eeq
The probability of the event can be made arbitrarily small by taking the limits $t\to\infty$ followed by $r\to\infty$, since the right-hand side is summable.
The claim follows.
\end{proof}
We now extend the estimate of the last lemma to include all $s\in [r,t-r]$, not only integer times. 
This is done using the fact that, if the event does happen at times $s$, the maximum at time $\lceil s \rceil$ is very likely to be still high, 
say greater than $m(\lceil s \rceil)+f_{t,\gamma/2}(\lceil s \rceil)$.

\begin{proof}[Proof of Theorem \ref{teor_upper_envelope}]
Plainly, the maximum at time $\lceil s \rceil$ of BBM is either smaller, or greater or equal to  $m(\lceil s \rceil)+f_{t,\gamma/2}(\lceil s \rceil)$.
Using this dichotomy on the event
\beq\label{anton.202}
\{\exists_{ s\in [r,t-r]} : \max_{k\leq n(s)}x_k(s)> y+m(s)+f_{t,\gamma}(s)\Big\}\ .
\eeq
we have that its probability is smaller or equal to 
\beq\label{anton.203}
\begin{aligned}
&\PP\Big[
\exists_{ s\in[r,t-r]}: \max_{k\leq n(s)} x_k(s)> y+m(s)+f_{t,\gamma}(s), \max_{k\leq n(\lceil s \rceil)}x_k(\lceil s \rceil)\geq y+m(\lceil s \rceil)+f_{t,\gamma/2}(\lceil s \rceil)
\Big]+\\
&\PP\Big[
\exists_{ s\in[r,t-r]}: \max_{k\leq n(s)} x_k(s)> y+m(s)+f_{t,\gamma}(s), \max_{k\leq n(\lceil s \rceil)}x_k(\lceil s \rceil)<y+m(\lceil s \rceil)+f_{t,\gamma/2}(\lceil s \rceil)
\Big].
\end{aligned}
\eeq
The first term is bounded above by
\beq\label{anton.204}
\PP
\Big[
\exists_{ s\in[r,t-r]}: \max_{k\leq n(\lceil s \rceil)}x_k(\lceil s \rceil)\geq y+m(\lceil s \rceil)+f_{t,\gamma/2}(\lceil s \rceil)
\Big]\ .
\eeq
By Lemma \ref{lem: integer time}, this is smaller than $\epsilon/2$ by choosing $r>r'(y,\gamma/2,\epsilon/2)$.
It remains to bound the second term.
Let $\mathcal{S}$ be the stopping time
\beq\label{anton.206}
\mathcal{S}\defi\inf\{s\in [1,t-1]: \max_{k\leq n(s)} x_k(s)> y+m(s)+f_{t,\gamma}(s)\}\ .
\eeq
(Here $1$ is just a choice to avoid technicalities at the endpoints).
By conditioning on $\mathcal{S}$, we can rewrite the probability as
\beq\label{anton.207}
\int_r^{t-r}\PP\Big(\max_{k\leq n(\lceil s' \rceil)}x_k(\lceil s' \rceil)<y+m(\lceil s' \rceil)+f_{t,\gamma/2}(\lceil s' \rceil)
\ \Big| \ \mathcal{S}=s'\Big)\ \PP(\mathcal{S}\in ds')
\eeq
We suppose $r>2$. Since $r>\lceil r \rceil-1$ and $t-r<\lceil t \rceil-\lceil r \rceil+1$, the above is smaller than
\beq\label{anton.208}
\sum_{j=\lceil r \rceil-1}^{\lceil t \rceil-\lceil r \rceil+1}\int_j^{j+1}\PP\Big(\max_{k\leq n(\lceil s' \rceil)}x_k(\lceil s' \rceil)<y+m(\lceil s' \rceil)+f_{t,\gamma/2}(\lceil s' \rceil)
\ \Big| \ \mathcal{S}=s'\Big)\ \PP(\mathcal{S}\in ds')\ .
\eeq
It remains to show that 
\beq\label{anton.209}
\PP\Big(\max_{k\leq n(\lceil s' \rceil)}x_k(\lceil s' \rceil)<y+m(\lceil s' \rceil)+f_{t,\gamma/2}(\lceil s' \rceil)
\ \Big| \ \mathcal{S}=s'\Big)
\eeq
tends to zero, uniformly in $s'$, as  $r\uparrow\infty$.  
But by the definition of $\mathcal{S}$, this probability is bounded by the probability that the offsprings at time $\lceil s' \rceil$ of the maximum at time $s'$
make a (downward) jump smaller than
\beq\label{anton.210}
m(\lceil s' \rceil)-m(s')+f_{t,\gamma/2}(\lceil s' \rceil)-f_{t,\gamma}(s')\ .
\eeq
By the Markov property of BBM, this is exactly
\beq\label{anton.211}
\PP\Big(\max_{k\leq n(\lceil s' \rceil-s')}x_k(\lceil s' \rceil-s')<m(\lceil s' \rceil)-m(s')+f_{t,\gamma/2}(\lceil s' \rceil)-f_{t,\gamma}(s')\Big)\ ,
\eeq
which, by Markov's inequality and the expected number of offsprings, is smaller than
\beq\label{anton.212}
e^{\lceil s' \rceil-s'}\ \PP\Big(x(\lceil s' \rceil-s')<m(\lceil s' \rceil)-m(s')+f_{t,\gamma/2}(\lceil s' \rceil)-f_{t,\gamma}(s')\Big)\ ,
\eeq
where $x$ is now a standard Brownian motion.
The first term is smaller than $e$ and $m(\lceil s' \rceil)-m(s')$ is smaller than $\sqrt{2}$. On the other hand, for $r\leq s'\leq t/2$,
\beq\label{anton.213}
f_{t,\gamma/2}(\lceil s' \rceil)-f_{t,\gamma}(s')=-{s'}^{\gamma}(1-\frac{\lceil s' \rceil^{\gamma/2}}{{s'}^{\gamma}})\leq -\frac{1}{2}r^\gamma\ ,
\eeq
where we chose $r$ large enough to get the factor $\frac{1}{2}$.
Similarly, for $t/2\leq s'\leq t-r$,
\beq\label{anton.214}
f_{t,\gamma/2}(\lceil s' \rceil)-f_{t,\gamma}(s)=-(t-s')^{\gamma}(1-\frac{(t-\lceil s' \rceil)^{\gamma/2}}{(t-s')^{\gamma}})\leq -\frac{1}{2}r^\gamma\ .
\eeq
Therefore,
\beq\label{anton.215}\begin{aligned}
\PP\Big(x(\lceil s' \rceil-s')<m(\lceil s' \rceil)-m(s')+f_{t,\gamma/2}(\lceil s' \rceil)-f_{t,\gamma}(s')\Big)\leq
\PP\big(x(\lceil s' \rceil-s')<\sqrt{2}-\frac{1}{2}r^\gamma\big)
\end{aligned}.\eeq
 Since $x(\lceil s' \rceil-s')$ is a Gaussian variable of variance $\lceil s' \rceil-s'\leq 1$, this probability
tends to zero as $r\uparrow \infty$, uniformly in $s'$.  The theorem follows.
\end{proof}

\subsection{The entropic envelope} \label{sec_entropic_repulsion}

In this section we present the proof of Theorem  \ref{teor_entropic_envelope}. 

\begin{proof} 
The claim of Theorem \ref{teor_entropic_envelope} is that, given a compact set,
 $D \subset \R$,
\beq \label{claim_entropic_zero} \bea
& \PP\Big[\exists k\leq n(t): x_k(t) \in m(t)+D, \; \text{but}\; \exists_{ s \in [r, t-r]}: x_k(s) \geq \overline{D} + E_{t, \alpha}(s)  \Big] \to 0
\eea\eeq
as $r \to\infty$, uniformly in $t> 3r$.

To see this, let $0<\gamma<\alpha<1/2$ and shorten $\underline{D} \defi \inf \{x\in D\}$. By Theorem \ref{teor_upper_envelope}, taking $y=\underline{D}$, paths of extremal particles must remain below the upper envelope for "most" of the time, hence it suffices to show that  
\beq \label{claim_entropic} \bea
& \PP\Big[\exists k\leq n(t): x_k(t) \in m(t)+D, \;  x_k(s) \leq \underline{D}+U_{t, \gamma}(s), \;\forall_{s \in [r, t-r]} \\
& \hspace{2cm} \text{but}\; \exists_{ s \in [r, t-r]}: x_k(s) \geq \overline{D} + E_{t, \alpha}(s)  \Big] \to 0
\eea\eeq
as $r \to\infty$, uniformly in $t> 3r$.  

Now, by Markov's inequality and using that $\E[n(t)]= e^t$, the probability in \eqref{claim_entropic} is at most 
\beq \label{claim_entropic_two} \bea
& \ee^t \PP\Big[x(t) \in m(t)+D, \;  x(s) \leq \underline{D}+U_{t, \gamma}(s),\;\forall_{ s \in [r, t-r]}, \text{but}\; \exists_{ s \in [r, t-r]}: x(s) \geq \overline{D}+E_{t, \alpha}(s) \Big] \\
&=\ee^t \PP\Big[x(t) \in m(t)+D,   x(s) \leq \underline{D}+\frac{s}{t} m(t) + f_{t,\gamma}(s), \forall_{s \in [r, t-r]} \\&\hspace{2cm}
\text{but} \; \exists_{ s \in [r, t-r]} : x(s) \geq \overline{D}+\frac{s}{t} m(t) - f_{t, \alpha}(s) \Big]. 
\eea\eeq
On the event $\{ x(t) \in m(t) +D \} =\{ x(t) -\overline{D} \leq m(t) \leq x(t) - \underline{D} \}$, we may replace the condition on the paths in the above probability by  
\beq\label{anton.17} 
x(s)  \leq \underline{D}+\frac{s}{t} \left\{ x(t)- \underline{D} \right\} + f_{t,\gamma}(s)  \leq \overline{D} \frac{t-s}{t}  + \frac{s}{t}  x(t) + f_{t,\gamma}(s), \quad \forall_{ s \in [r, t-r]},
\eeq
and 
\beq\label{anton.18} 
  \exists_{ s \in [r, t-r]}:x(s)  \geq \overline{D}+\frac{s}{t}\left\{x(t) -\overline{D}\right\} - f_{t, \alpha}(s) 
 = \overline{D} \frac{t-s}{t} + \frac{s}{t} x(t) - f_{t, \alpha}(s) 
 \eeq
to get that \eqref{claim_entropic_two} is at most
\beq \label{claim_entropic_three} \bea
& \ee^t \PP\Big[x(t) \in m(t)+D, \;  x(s) -\frac{s}{t} x(t) \leq  \overline{D} \frac{t-s}{t}  + f_{t,\gamma}(s)  \;\forall_{ s \in [r, t-r]} \\
& \hspace{2cm} \text{but}\; \; \exists_{ s \in [r, t-r]}: x(s)- \frac{s}{t} x(t) \geq  \overline{D} \frac{t-s}{t}  - f_{t, \alpha}(s)  \Big] \\
&  =  \ee^t \PP\left[x(t) \in m(t)+D\right]  \\
&  \times \PP\Big[ \mathfrak z_t(s) \leq  \overline{D} \frac{t-s}{t}  + f_{t,\gamma}(s),\forall_{ s \in [r, t-r]},
 \text{but} \; \exists_{ s \in [r, t-r]} : \mathfrak z_t(s) \geq \overline{D} \frac{t-s}{t}  - f_{t, \alpha}(s) \Big],
\eea\eeq
where the last step uses  the independence of $x(t)$ and $\mathfrak z_t(s), \; 0\leq s \leq t$. Since 
\beq\label{anton.19}
\ee^t \PP\left[x(t) \in D+m(t)\right] \leq \kappa  t  \int_{D} \exp\left[-\sqrt{2} x\right] \dd x,
\eeq
for some $\kappa>0$ and $t\geq 2$, Theorem \ref{teor_entropic_envelope} will follow as soon as we prove that 
\beq \bea \label{claim_entropic_four}
& t   \PP\Big[ \mathfrak z_t(s) \leq \overline{D} \frac{t-s}{t}  + f_{t,\gamma}(s)  \;\forall_{ s \in [r, t-r]},\\
& \hspace{2cm} \text{but} \; \exists_{ s \in [r, t-r]}: \mathfrak z_t(s) \geq  \overline{D} \frac{t-s}{t}  - f_{t, \alpha}(s)  \Big] \to 0,
\eea \eeq
as $r\to \infty$ uniformly in $t> 3r$. To see this, we  observe that 
\beq\label{anton.20} \bea
& \left\{ \exists_{ s \in [r, t-r]} :  \mathfrak z_t(s) \geq  \overline{D} \frac{t-s}{t}  - f_{t, \alpha}(s)   \right\}^c 
  \subseteq   \left\{ \mathfrak z_t(s) \leq \overline{D} \frac{t-s}{t}  + f_{t,\gamma}(s)  \;\forall_{ s \in [r, t-r]} \right\}.
\eea \eeq
Hence the left hand side of  \eqref{claim_entropic_four} equals
\beq \bea \label{claim_entropic_five}
& t\Bigg( \PP\left[  \mathfrak z_t(s) \leq \overline{D} \frac{t-s}{t}  + f_{t,\gamma}(s)  \;\forall_{ s \in [r, t-r] } \right]-\PP\left[ \mathfrak z_t(s) \leq  \overline{D} \frac{t-s}{t}  - f_{t, \alpha}(s), \forall_{ s \in [r, t-r]}   \right]\Bigg).
\eea \eeq
Define the functions
\beq\label{anton.21}
f(s) \defi  \overline{D} \frac{t-s}{t},
\eeq
\beq\label{anton.22}
F(s) \defi f(s) + f_{t,\gamma}(s),
\eeq
\beq\label{anton.24}
\overline{F}(s) \defi f(s) + f_{t, \alpha}(s), 
\eeq
and 
\beq\label{anton.23}
\underline{F}(s) \defi f(s) - f_{t, \alpha}(s).
\eeq
With the notations of Section \ref{sec_bb_properties}, we may then introduce the probability $P^0\left[ B^0[r, t-r]\right]$  and rewrite \eqref{claim_entropic_five} as 
\beq \bea \label{claim_entropic_six}
& t \left(P^0\left[ B^F[r, t-r]\right] - P^0\left[ B^{\underline F}[r, t-r]\right]\right) \\
& \qquad = 
t  P^0\left[ B^0[r, t-r]\right] \frac{P^0\left[ B^F[r, t-r]\right]}{P^0\left[ B^0[r, t-r]\right]} 
 \left( 1- \frac{P^0\left[ B^{\underline F}[r, t-r]\right]}{P^0\left[ B^F[r, t-r]\right]} \right).
\eea \eeq
We clearly have $\underline{F} \leq F \leq \overline{F}$. 
Moreover, we can pick $r$ large enough so that
$\underline{F}\leq 0\leq F$ on $[r,t-r]$. From Lemma \ref{monotonicity} we deduce
\beq\label{anton.25}
P^0\left[ B^{\underline F}[r, t-r]\right] \leq P^0\left[ B^0[r, t-r]\right] \leq P^0\left[ B^F[r, t-r]\right] \leq P^0\left[ B^{\overline{F}}[r, t-r]\right], 
\eeq
and therefore
\beq\label{anton.26}
\frac{P^0\left[ B^{\underline F}[r, t-r]\right]}{P^0\left[ B^{\overline{F}}[r, t-r]\right]} \leq \frac{P^0\left[ B^{\underline F}[r, t-r]\right]}{ P^0\left[ B^F[r, t-r]\right]}\leq 1\leq \frac{P^0\left[ B^{F}[r, t-r]\right]}{ P^0\left[ B^0[r, t-r]\right]}\leq \frac{P^0\left[ B^{\overline{F}}[r, t-r]\right]}{P^0\left[ B^{\underline F}[r, t-r]\right]} \ .
\eeq
This, in particular, entails that \eqref{claim_entropic_six} is at most 
\beq \label{claim_entropic_seven}
 t  P^0\left[ B^0[r, t-r]\right] \frac{P^0\left[ B^{\overline{F}}[r, t-r]\right]}{P^0\left[ B^{\underline F}[r, t-r]\right]} \left(1- \frac{P^0\left[ B^{\underline F}[r, t-r]\right]}{P^0\left[ B^{\overline{F}}[r, t-r]\right]} \right).
\eeq
By symmetry of the Brownian Bridge around the $x$-axis, we may rewrite 
\beq \bea \label{symmetry_bbridge}
\frac{P^0\left[ B^{\underline F}[r, t-r]\right]}{P^0\left[ B^{\overline{F}}[r, t-r]\right]} & = \frac{P^0\left[ B_{-\underline F}[r, t-r]\right]}{P^0\left[ B_{-\overline{F}}[r, t-r]\right]} \\
& = \frac{\PP\left[ \mathfrak z_t(s) > -f(s) + f_{t,\alpha}(s),\quad  r\leq s \leq t-r \right]}{\PP\left[ \mathfrak z_t(s) > -f(s) - f_{t,\alpha}(s), \quad r\leq s \leq t-r  \right]}. 
\eea \eeq
The goal is thus to show that \eqref{symmetry_bbridge} converges to $1$ as $r\to \infty$ uniformly in $t$. Some caution is however
needed: a simple inspection of the bounds in Lemma \ref{lem: BB estimate} shows that 
\beq \label{exploding_t_r}
\lim_{r\to \infty} \lim_{t\to\infty} t P^0\left[ B^0[r, t-r]\right] = +\infty, 
\eeq
so to prove that \eqref{claim_entropic_seven} indeed converges to $0$ as $r\to\infty$ uniformly in $t$, 
we need uniform bounds to \eqref{symmetry_bbridge} which compensate \eqref{exploding_t_r}. 
To do this, we  follow Bramson \cite[Section 6, pp.85-88]{bramson_monograph} 
 and use a Cameron-Martin or Girsanov change of measure. 

Let $P_a$ be the law of the Brownian bridge on $[0,t]$ with drift $a(s)\in L^2[0,t]$.
By Girsanov's formula, for any Borel set 
$A\subset C[0,t]$,
\beq\label{anton.27}
P_a[A] = \E\left[\exp\left( \int_0^t a(s) \dd \mathfrak z_t(s) -\frac{1}{2} \int_0^t a^2(s) \dd s +\frac{1}{2t} \left\{\int_0^t a(s) \dd s\right\}^2\right); A   \right]  ,
\eeq
This allows to deform the curve $-f+f_{t,\alpha}$ in the numerator of the r.h.s. of \eqref{symmetry_bbridge} into 
$-f-f_{t, \alpha}$ appearing in the denominator: 
\beq \bea \label{deforming_one}
& \PP\left[ \mathfrak z_t(s) > -f(s) + f_{t,\alpha}(s),\quad  r\leq s \leq t-r \right] \\
& \qquad  = \PP\left[ \mathfrak z_t(s) -\beta_{r,t}(s) >- f(s) - f_{t,\alpha}(s),\quad  r\leq s \leq t-r \right],
\eea \eeq
where
\beq \label{defi_beta}
\beta_{r,t}(s) \defi \begin{cases}
2 r^{\alpha-1} s & 0\leq s \leq r \\
2 s^\alpha & r\leq s \leq t/2 \\
2 (t-s)^\alpha & t/2 \leq s \leq t-r\\
2 r^{\alpha-1}(t-s) & t-r \leq s \leq t. 
\end{cases}
\eeq
(Assuming, say, that $t> 3r$). The process $\mathfrak z_t(s) -\beta_{r,t}(s)$
 is a diffusion with drift
$a(s)=-\beta'_{r,t}(s)$. Therefore, by change of measure, we may rewrite \eqref{deforming_one}
as
\beq
 \label{equal_two}
 \bea
& \E\left[ \exp\left( - \int_0^t \beta_{r,t}'(s) \dd \mathfrak z_t 
-\frac{1}{2} \int_0^t \beta_{r,t}'(s)^2 \dd s - \underbrace{\frac{1}{2t}\left[\int_0^t \beta_{r,t}'(s) \dd s\right]^2}_{= 0}
\right); \; A_{r,t}\right] \\
& \qquad = \E\left[ \exp\left( - \int_0^t \beta_{r,t}'(s) \dd \mathfrak z_t 
- \frac{1}{2} \int_0^t \beta_{r,t}'(s)^2 \dd s \right); \;  A_{r,t}\right],
\eea
\eeq
where $A_{r,t}$ is the event
\beq\label{anton.28}
A_{r,t} \defi \left\{\mathfrak z: \; \mathfrak z_t(s) >- f(s) - f_{t, \alpha}(s), \; s \in [r, t-r] \right\}
\eeq
To analyse this, we consider the subset
\beq\label{anton.29}
A_{r,t}^1 \defi A_{r,t} \cap \left\{\mathfrak z: \; \mathfrak z_t(s)<\Lambda_t(s), \; r\leq s \leq t-r\right\}
\eeq
with 
\beq\label{anton.30}
\Lambda_t(s) \defi \begin{cases}
2 s^\theta & 0\leq s \leq t/2 \\
2 (t-s)^\theta & t/2 \leq s \leq t,
\end{cases}
\eeq
and $\frac{1}{2}< \theta < 1-\alpha$. To control the  behavior on $A_{r,t}^1$,
note that $\beta_{r,t}'(s)$ is decreasing in $s$ and constant on $[0,r]$ and $[t-r,t]$. 
Define $l_{r,t}(s)$ as in \eqref{defi_beta}, but with $\alpha$ replaced by $\theta$.  
We now want to estimate $\int_0^t \beta_{r,t}'(s) \dd \mathfrak z_t(s)$.
Using the special form of the function $\beta_{r,t}(s)$, one sees, using 
integration  by parts, that 
\beq\label{anton.31}\bea 
\int_0^t \beta_{r,t}'(s) \dd \mathfrak z_t(s) & = -\int_r^{t-r} \beta_{r,t}''(s) \mathfrak z_t(s)\dd s \leq  -\int_r^{t-r} \beta_{r,t}''(s) \Lambda_t(s)\dd s\ ,
\eea \eeq
where the last inequality follows since  $-\beta''\geq 0$ and $\mathfrak z_t(s)\leq \Lambda_t(s)$ . Using again integration by parts on the last term in 
(\ref{anton.31}), we arrive at 
\beq\label{anton.32}
- \int_0^t \beta_{r,t}'(s) \dd \mathfrak z_t(s) \geq -\int_0^t \beta_{r,t}'(s) \dd l_{r,t}(s). 
\eeq
A simple computation shows that  the r.h.s. above is bounded from below by
 $-\kappa_1 r^{\alpha+\theta-1}$, for some positive constant $\kappa_1>0$. 

Another simple computation gives that
\beq\label{anton.33}
-\frac{1}{2} \int_0^t \beta_{r,t}'(s)^2 \dd s \geq -\kappa_2 r^{2\alpha-1},
\eeq
for another constant $\kappa_2>0$. Consequently, \eqref{equal_two} is bounded from below by
\beq \label{equal_three}
\exp[- \kappa_3 r^{\alpha+\theta-1}] \PP\left[A_{r,t}^1 \right]
\eeq
for some constant $\kappa_3>0$. Notice  that for our choice, $\alpha+\theta-1$ is  negative. Hence \eqref{equal_three} approaches 
$\PP\left[A_{r,t}^1 \right]$ for large $r$. 

We have thus reduced the problem to showing that
\beq \label{equal_four}
\PP\left[A_{r,t}^1 \right]\Big/\PP\left[A_{r,t} \right] \to 1,
\eeq
 as $r\to \infty$ uniformly in $t$. To do so, observe that
\beq \label{equal_five}\bea 
\PP\left[A_{r,t}^1 \right]\Big/\PP\left[A_{r,t} \right] & = \PP\left[ \mathfrak z_t(s)<\Lambda_t(s), \; r\leq s \leq t-r \Big| A_{r,t} \right].
\eea \eeq
Recalling that $f(s) = \frac{t-s}{t} \overline{D}$, we see that there exists a (finite) $r_o = r_o(\overline{D}, C, \alpha)$ such that    
\beq\label{anton.34}
-f(s) - f_{t, \alpha}(s) < 0 \quad \text{for}\quad r\leq s \leq t-r,
\eeq
as soon as $r\geq r_o$. We can therefore use  the monotonicity property \eqref{below_lambda} from Lemma \ref{monotonicity}, which ensures that,
for all $r\geq r_o$, the r.h.s. in \eqref{equal_five} is not smaller than
\beq\label{anton.35} \bea
& \PP\left[ \mathfrak z_t(s)<\Lambda_t(s), \; r\leq s \leq t-r \Big|  \mathfrak z_t(s)> 0, \; r\leq s \leq t-r  \right] = P^0\left[ B^{\Lambda_t}[r, t-r] \big| B_0[r, t-r] \right],
\eea \eeq
which, by  \eqref{more_bound} from Lemma \ref{more_than_fluctuations}, is in turn bounded from below by
\beq \label{at_least} \bea
& 1- \kappa_4 \sum_{k=r}^\infty k \exp\left[- C k^{2\theta-1}\right] = 1-\kappa_4 \sum_{k=r}^\infty k \exp\left[- C k^{1-2\alpha}\right]
\eea \eeq
for some  $\kappa_4>0$. \\

\eqref{at_least} provides a uniform lower bound on \eqref{symmetry_bbridge}. 
Therefore, \eqref{claim_entropic_seven} is, up to irrelevant numerical constants, smaller than 
\beq \bea \label{claim_entropic_eight}
& t  P^0\left[ B^0[r, t-r]\right]   \frac{\sum_{k=r}^\infty k \exp\left[- C k^{1-2\alpha}\right]}{1-\sum_{k=r}^\infty k \exp\left[- C k^{1-2\alpha}\right]} \\
& \quad \leq t  P^0\left[ B^0[r, t-r]\right]  
\frac{\int_r^\infty x \exp\left[- C x^{1-2\alpha}\right] \dd x }{1-\int_r^\infty x \exp\left[- C x^{1-2\alpha}\right] \dd x}\\
& \stackrel{\text{Lem.} \ref{lem: BB estimate}}{\leq} t\frac{2r}{t-2r}  \frac{\int_r^\infty x \exp\left[- C x^{1-2\alpha}\right] \dd x}{1-\int_r^\infty x \exp\left[- C x^{1-2\alpha}\right] \dd x}.
\eea \eeq
Since $\alpha<1/2$, the integral appearing in the last term converges  to $0$, as $r\to \infty$. This implies that 
\eqref{claim_entropic_eight} converges to $0$, as $r\to \infty$, uniformly in $t$. This concludes the  proof. 
\end{proof}

\subsection{The lower envelope} \label{sec_lower_envelope}

\begin{proof}[Proof of Theorem \ref{teor_lower_envelope}]
Recall that we are given a compact set $D\subset \R$, $E_{t, \beta}$ a curve such as in \eqref{def_entropic_envelope} (with parameters $1/2< \beta < 1$). 
Theorem \ref{teor_lower_envelope} asserts that, with $\overline{D} \defi \sup\{x\in D\}$,  
\beq \label{claim_lower_one} \bea
& \PP\Big[\exists k\leq n(t):\; x_k(t)\in m(t)+D, \; \text{but}  \exists_{ r \leq s \leq t-r}: x_k(s) \leq \overline{D}+ E_{t, \beta}(s)  \Big]\downarrow 0,  
\eea \eeq 
as $r\to \infty$, uniformly in $t> 3r$. 

To see this, let $\alpha$ such that $0< \alpha < 1/2 < \beta$ and set $\underline{D} \defi \inf\{x: \, x\in D\}$. To prove \eqref{claim_lower_one}, 
by Theorem \ref{teor_entropic_envelope} it suffices to prove that 
\beq \bea \label{claim_lower_one_one}
& \PP\Big[\exists k\leq n(t):\, x_k(t)\in m(t)+D, \; x_k(s) \leq 
\underline{D}+E_{t,\alpha}(s)\;  \forall_{ s \in [r,t-r]}, \\
& \hspace{5cm} \text{but} \exists_{s \in [r,t-r]}:  x_k(s) \leq 
\overline{D}+ E_{t, \beta}(s) \Big] \downarrow 0,
\eea
\eeq
as $r\to \infty$, uniformly in $t> 3r$.

Now, by Markov's inequality, and using  that, on the event  $\{x(t) \in D+ m(t)\}$, one has that
$x(t)- \overline{D} \leq m(t) \leq x(t) -\underline{D}$, one gets that the probability in \eqref{claim_lower_one_one} is
at most  
\beq \label{claim_lower_two} \bea
& \ee^t \PP\Big[x(t)\in m(t)+D\Big]   \PP\Big[ \mathfrak z_t(s) \leq \left\{\overline{D}- \frac{s}{t}\underline{D} \right\}- f_{t,\alpha}(s), \forall_{r \leq s \leq  t-r}, \\
& \hspace{2cm} \text{but} \; \exists_{ r \leq s \leq t-r}:\mathfrak z_t(s) \leq \left\{\overline{D}- \frac{s}{t}\underline{D} \right\}- f_{t,\beta}(s)  \Big],
\eea \eeq  
where we have used that $\mathfrak z_t(s) \defi x(t) -\frac{s}{t}x(s)\, (0\leq s \leq t)$, is a Brownian bridge independent of $x(t)$.
Clearly, 
\beq\label{anton.43}
 \ee^t \PP\Big[x(t)\in m(t)+D\Big]  \leq \kappa_1 t \int_D \ee^{-\sqrt{2} x} \dd x,
\eeq
for some $\kappa_1 >0$ ($t\geq 2$); moreover,  with $\text{diam}(D) \defi  |\overline{D}|+ |\underline{D}|$, the second factor in 
\eqref{claim_lower_two} is bounded from above by
\beq\label{anton.44} \bea
\PP\Big[ \mathfrak z_t(s) \leq \text{diam}(D) - f_{t,\alpha}(s),\forall_{r \leq s \leq  t-r}, \text{but}\;  \exists_{ r \leq s \leq t-r}: \mathfrak z_t(s) \leq \text{diam}(D) - f_{t,\beta}(s) \Big].
\eea \eeq
Thus, Theorem \ref{teor_lower_envelope} will follow as soon as we prove that 
this term multiplied by $t$ tends to zero, as $r\uparrow \infty$, uniformly in 
$t>3r$.

Let $0<a<1$ such that $2a\beta-1>0$. To given compact $D \subset \R$ we may find $\tilde r = \tilde r(\alpha, \beta, D,a)$ such that for $r\geq \tilde r$ one has 
$\text{diam}(D) - f_{t, \alpha}(s) \leq 0$, as well as
$\text{diam}(D) -  f_{t,\beta}(s) \leq -f_{t, a\beta}(s)$ for all $r\leq s \leq t-r$ . 
 Hence, $t$ times \eqref{claim_lower_two} is at most 
\beq \label{claim_lower_four} \bea
& t \PP\Big[ \mathfrak z_t(s) \leq 0, \forall_{r \leq s \leq  t-r}, 
 \text{but}\;  \exists_{ r \leq s \leq t-r}: \mathfrak z_t(s) \leq - f_{t,a\beta}(s)  \Big] \\
& \qquad = t \PP\Big[ \mathfrak z_t(s) \geq 0, \forall_{r \leq s \leq  t-r},  \text{but}\; \exists_{ r \leq s \leq t-r}: \mathfrak z_t(s) \geq f_{t,a\beta}(s)  \Big],
\eea \eeq
where the  last equality is due to the  symmetry of the Brownian bridge around the $x$-axis. But \eqref{claim_lower_four} equals 
\beq \label{claim_lower_five} \bea
& t \left(\PP\Big[ \mathfrak z_t(s) \geq 0 \forall_{ r \leq s \leq  t-r}\Big] - \PP\left[ 0 \leq  \mathfrak z_t(s) \leq f_{t,a\beta}(s) \forall_{r \leq s \leq t-r} \right] \right) \\
& \qquad = t P^0\left[ B_0[r, t-r] \right] \left\{1 -P^0\left[ B^{f_{t, a\beta}}[r \leq s \leq t-r] \Big| B_0 \right] \right\} \\
& \qquad \stackrel{\eqref{more_bound}}{\leq} \kappa  t P^0\left[ B_0[r, t-r] \right]  \int_{r}^\infty x \ee^{- x^\delta} \dd x,
\eea \eeq
with $\kappa>0$ and $\delta \defi 2a\beta -1$. The last factor tends to zero faster than any power, as $r\uparrow \infty$. By 
Lemma \ref{lem: BB estimate},
\beq\label{anton.45}
tP^0\left[ B_0[r, t-r] \right] \leq \kappa \frac{rt}{t-2 r},
\eeq
which is smaller than $\kappa r$, if $t>3r$. Hence the right-hand side of 
(\ref{claim_lower_five}) tends to zero with $r$, uniformly on $t>3r$, which 
proves the theorem.
\end{proof}

{\bf Acknowledgments.} The list of people we are indebted to for useful conversations on various issues related to this work is long: in particular, we wish to express our gratitude to D. Aldous, E. Bolthausen, M. Bramson, J. Berestycki, P. L. Ferrari, D. Gruhlke, S.C. Harris, D. Ioffe, L. Zambotti,
 and O. Zeitouni. We also thank the anonymous referee for the careful reading and the insightful observations which have led to an improvement of the paper.

\end{document}